\documentclass[11pt, reqno]{amsart}
\usepackage{hyperref}
\usepackage{amsmath}
\usepackage{amscd}
\usepackage{amsthm}
\usepackage{amssymb}
\usepackage{latexsym}
\usepackage{eufrak}
\usepackage{euscript}
\usepackage{epsfig}
\usepackage{graphics}
\usepackage{array}
\usepackage{enumerate}
\usepackage{dsfont}
\usepackage{color}
\usepackage{wasysym}
\usepackage{graphicx}
\usepackage{amsfonts}
\usepackage{mathrsfs}
\usepackage{dsfont}
\usepackage{indentfirst}
\usepackage{esint}

\newcommand{\R}{{\mathbb R}}

\newcommand{\supp}{{\mathrm{supp}\,}}

\newcommand{\diam}{\mathrm{diam}}

\newtheorem{thm}{Theorem}[section]
\newtheorem{coro}[thm]{Corollary}
\newtheorem{lemma}[thm]{Lemma}
\newtheorem{pro}[thm]{Proposition}

\theoremstyle{definition}

\newtheorem{example}[thm]{Example}
\newtheorem{remark}[thm]{Remark}

\newcommand{\Hmm}[1]{\leavevmode{\marginpar{\tiny%
$\hbox to 0mm{\hspace*{-0.5mm}$\leftarrow$\hss}%
\vcenter{\vrule depth 0.1mm height 0.1mm width \the\marginparwidth}%
\hbox to 0mm{\hss$\rightarrow$\hspace*{-0.5mm}}$\\\relax\raggedright
#1}}}

\begin{document}
\title[Multi-way dual Cheeger constants]{Multi-way dual Cheeger constants and spectral bounds of graphs}
\author{Shiping Liu}
\address{Department of Mathematical Sciences, Durham University, DH1 3LE Durham, United Kingdom}
\email{shiping.liu@durham.ac.uk}
\begin{abstract}
We introduce a set of multi-way dual Cheeger constants and prove universal higher-order dual Cheeger inequalities for eigenvalues of normalized Laplace operators on weighted finite graphs. Our proof proposes a new spectral clustering phenomenon deduced from metrics on real projective spaces. We further extend those results to a general reversible Markov operator and find applications in characterizing its essential spectrum.

\end{abstract}

\maketitle
\section{Introduction and main ideas}
Cheeger constant, encoding global connectivity properties of the underlying space, was invented by Cheeger \cite{Cheeger1970} and related to the first non-zero eigenvalue of the Laplace-Beltrami operator on a compact Riemannian manifold, which is now well-known as Cheeger inequality. Afterwards, it was extended to discrete settings by several authors in spectral graph theory or Markov chain theory, see e.g. \cite{Dodziuk1984}, \cite{AM1985}, \cite{Alon1986}, \cite{LS1988}, \cite{Mohar88}, \cite{SJ1989}, \cite{DS1991}, \cite{Chung96}. Intriguingly, this stimulated research in many unexpected theoretical or practical areas, such as the explicit construction of expander graphs, see e.g. \cite{Alon1986}, \cite{Lubotzky1994}, \cite{Tanner1984}, \cite{PA2012}, graph coloring, image segmentation, web search, approximate counting, for which we refer to \cite{LGT2013}, \cite{KLLGT2013} for detailed references.

Recently, Miclo \cite{Miclo2008} (see also \cite{DJM2012}) introduced a set of multi-way Cheeger constants (alternatively called higher-order isoperimetric constants), $h(k)$, $k=1,2,\ldots$, in discrete setting and conjectured related higher-order Cheeger inequalities universal for any weighted graph. This conjecture was solved by Lee, Oveis Gharan and Trevisan \cite{LGT2013} by bringing in the powerful tool of random metric partitions developed originally in theoretical computer science. Moreover, their approach justifies the empirical spectral clustering algorithms in \cite{NJW2002} which are very popular and powerful tools in many practical fields (see e.g. \cite{Luxburg07}, \cite{Ruben13}). Amazingly, this new progress of spectral graph theory provides feedback to the setting of Riemannian manifold. Funano \cite{Funano2013} and Miclo \cite{Miclo2013} (with different strategies) extended the higher-order Cheeger inequalities to weighted Riemannian manifolds and found very important applications there.

In contrast to a Riemannian manifold, on a graph, the spectrum of the normalized Laplace operator is bounded from above by $2$. Explicitly, one can list them as
\begin{equation}
0=\lambda_1\leq\lambda_2\leq\cdots\leq\lambda_{N}\leq 2,
\end{equation}
where $N$ is the size of the graph $G$.
Therefore, a graph has its own particular spectral gaps $2-\lambda_k$ which have no counterparts in the Riemannian setting. In order to investigate the spectral gap $2-\lambda_N$, Bauer and Jost \cite{BJ13} introduced a dual Cheeger constant, $\overline{h}(1)$ in our notation below, encoding the bipartiteness property of the underlying graph. (Independently, a related constant called bipartiteness ratio on regular graphs was studied by Trevisan \cite{Trevisan2012}.) Explicitly, it holds that
\begin{equation}
 \text{connected }G \text{ is bipartite }\Leftrightarrow\text{ } \overline{h}(1)=1.
\end{equation}
(A graph is bipartite if its vertex set can be divided into two classes and edges are only permitted between two vertices from opposite classes.)
They then proved a dual Cheeger inequality, providing a strong quantitative version of the fact that $2-\lambda_N$ vanishes if and only if the underlying graph is bipartite. This has already found important applications for the convergence of random walks on graphs, synchronization for coupled map lattices \cite{BJ13} and characterizing behaviors of the essential spectrum of infinite graphs \cite{BHJ2012}.

In this paper, we introduce a set of multi-way dual Cheeger constants, $\overline{h}(k)$, $k=1,2,\ldots, N$, encoding more detailed information about how far/close a graph is from being a bipartite one. The dual relations between $\overline{h}(k)$ and $h(k)$ are manifested by the fact that
\begin{equation}
 G \text{ is bipartite }\Leftrightarrow\text{ } h(k)+\overline{h}(k)=1, \,\,1\leq k\leq N.
\end{equation}
In fact, if a graph can satisfy $h(k)+\overline{h}(k)=1$ for a small number $k$, then, roughly speaking, it actually has a large size bipartite subgraph in a reasonable sense (Proposition \ref{pro1} (\textrm{iii})). For example, it holds that for an odd cycle $\mathcal{C}_N$ (Proposition \ref{proOddcycle}), $$h(k)+\overline{h}(k)=1,\,\,\,2\leq k\leq N.$$ Recall that an odd cycle is not bipartite but very close to be bipartite. Moreover, this framework provides a new viewpoint about the previously defined constants. We see that the dual Cheeger constant of Bauer-Jost, $\overline{h}(1)$, is actually dual to $h(1)=0$ and the Cheeger constant $h(2)$ is dual to $\overline{h}(2)$.

We prove higher-order dual Cheeger inequalities, i.e., we derive estimates for the spectral gaps $2-\lambda_{N-k+1}$ in terms of $\overline{h}(k)$, which hold universally for any weighted finite graph (see Theorem \ref{TheoremMain}). This completes the picture about graph spectra and (dual) isoperimetric constants. Interestingly, our proof proposes a new type of spectral clustering via the top $k$ eigenfunctions employing metrics on real projective spaces. As in \cite{LGT2013}, the proof is in principle algorithmic and hence we anticipate the practical applications of this new spectral clustering.

The deep relations between higher eigenvalues and geometry of graphs have been explored in the works of Chung, Grigor'yan and Yau \cite{CY1995, CGY1996, CGY2000}. For discussions about the spectral gap $2-\lambda_N$ and curvature of graphs, we refer the readers to \cite{BJL2012}. In Markov chain theory, there is the fundamental work of Diaconis and Stroock \cite{DS1991} about various geometric bounds of eigenvalues, in particular, of $2-\lambda_N$. Note that the language of Markov chains and that of normalized graph Laplacian we use here can be translated into each other. For example, a chain is aperiodic if and only if its associated graph is not bipartite .

In this spirit, it turns out that our results can be applied to a very general setting. We extend the multi-way dual Cheeger constants and higher-order dual Cheeger inequalities to a reversible Markov operator $P$ on a probability space $(X, \mathcal{F}, \mu)$, following recent works of Miclo \cite{Miclo2013} and F.-Y. Wang \cite{Wang2014}. Let us denote the infimum (supremum, resp.) of the essential spectrum of $P$ by $\overline{\lambda}_{\text{ess}}(P)$ ($\lambda_{\text{ess}}(P)$, resp.). We obtain a characterization for $\overline{\lambda}_{\text{ess}}$  in terms of extended multi-way dual Cheeger constants $\overline{h}_P(k)$,
$$\overline{\lambda}_{\text{ess}}(P)>-1 \Leftrightarrow \inf_{k\geq 1}\overline{h}_P(k)<1.$$
It can be considered as the counterpart of F.-Y. Wang's new criterion for $\lambda_{\text{ess}}$ in terms of multi-way Cheeger constants $h_P(k)$,
$$\lambda_{\text{ess}}(P)<1 \Leftrightarrow \sup_{k\geq 1}h_P(k)>0.$$
Both arguments employ an approximation procedure developed by Miclo, by which he solves the conjecture of Simon and H{\o}egh-Krohn \cite{SH1972} in a semi-group context. A further discussion about the relations between $h_P(k)$ and $\overline{h}_P(k)$ enables us to arrive at
$$\sup_{k\geq 1}h_P(k)>0 \Leftrightarrow -1<\overline{\lambda}_{\text{ess}}(P)\leq \lambda_{\text{ess}}(P)<1.$$

\subsection{Statements of main results}
In order to put our results into perspective, we start with recalling the (higher-order) Cheeger inequalities.
Let $G=(V, E, w)$ be an undirected, weighted finite graph without self-loops. $V$ and $E$ stands for the set of vertices and edges, respectively.  We denote by $w_{uv}$ the positive symmetric weight associated to $u, v\in V$, where $e=\{u,v\}\in E$ (sometimes we also write $u\sim v$). For convenience, we may put $w_{uv}=0$ if $u, v$ are not connected by an edge. The degree $d_u$ of a vertex $u$ is then defined as $d_u:=\sum_{v,v\sim u}w_{uv}$.

The expansion (or conductance) of any non-empty subset $S\subseteq V$ is defined as
\begin{equation*}
\phi(S)=\frac{|E(S, \overline{S})|}{\text{vol}(S)},
\end{equation*}
where $\overline{S}$ represents the complement of $S$ in $V$, and $|E(S, \overline{S})|:=\sum_{u\in S, v\in \overline{S}}w_{uv}$, $\text{vol}(S):=\sum_{u\in S}d_u=|E(S, S)|+|E(S, \overline{S})|$.

Then, for every $k\in \mathbb{N}$, the $k$-way Cheeger constant is defined as
\begin{equation}\label{Multiway Cheeger constant}
h(k)=\min_{S_1, S_2,\ldots, S_k}\max_{1\leq i\leq k}\phi(S_i),
\end{equation}
where the minimum is taken over all collections of $k$ non-empty, mutually disjoint subsets $S_1, S_2, \ldots, S_k\subseteq V$. We call such kind of $k$ subsets a $k$-subpartition of $V$, following \cite{DJM2012}. Note by definition, we
have the monotonicity $h(k)\leq h(k+1)$. The classical Cheeger inequality asserts that
\begin{equation}\label{classicalCheeger}
 \frac{h(2)^2}{2}\leq\lambda_2\leq 2h(2).
\end{equation}

Resolving a conjecture of Miclo \cite{Miclo2008} (see also \cite{DJM2012}), Lee-Oveis Gharan-Trevisan \cite{LGT2013} prove the following higher-order Cheeger inequality.
\begin{thm}[Lee-Oveis Gharan-Trevisan]\label{TheoremLGT}
For every graph $G$, and each natural number $1\leq k\leq N$, we have
\begin{equation}\label{LGT1}
\frac{\lambda_k}{2}\leq h(k)\leq Ck^2\sqrt{\lambda_k},
\end{equation}
or in another form
\begin{equation}\label{LGT2}
\frac{1}{C^2k^4}h(k)^2\leq\lambda_k\leq 2h(k),
\end{equation}
where $C$ is a universal constant.
\end{thm}

Observe that when $k>\frac{N}{2}$, at least one of $k$ disjoint non-empty subsets must contain a single vertex, hence $h(k)=1$. Therefore (\ref{LGT2}) is more useful for the first half part of the spectrum.

We will study the corresponding phenomena for the remaining part of the spectrum. Define the following quantity for a pair of disjoint subsets $V_1, V_2\subseteq V$, for which $V_1\cup V_2\neq \emptyset$,
\begin{equation*}
\overline{\phi}(V_1, V_2)=\frac{2|E(V_1, V_2)|}{\text{vol}(V_1\cup V_2)}.
\end{equation*}
Then, for every $k\in \mathbb{N}$, we can define a $k$-way dual Cheeger constant as follows.
\begin{equation}\label{Multiway dual Cheeger constant}
\overline{h}(k)=\max_{(V_1, V_2),\ldots, (V_{2k-1}, V_{2k})}\min_{1\leq i\leq k}\overline{\phi}(V_{2i-1}, V_{2i}),
\end{equation}
where the maximum is taken over all collections of $k$ pairs of subsets $$(V_1, V_2),(V_3, V_4)\ldots, (V_{2k-1}, V_{2k})$$
which satisfy $$V_p\cap V_q=\emptyset, \,\,\forall\,\, 1\leq p\neq q \leq 2k, V_{2i-1}\cup V_{2i}\neq \emptyset, \,\,\forall\,\, 1\leq i\leq k.$$
For notational simplicity, we will denote the space of all $k$ pairs of subsets described as above by $\text{Pair}(k)$ and call every element of $\text{Pair}(k)$ a $k$-sub-bipartition of $V$. Here we have the monotonicity
$\overline{h}(k)\geq \overline{h}(k+1)$.

Bauer-Jost \cite{BJ13} proved a dual Cheeger inequality
\begin{equation}\label{BJdualCheeger}
 \frac{(1-\overline{h}(1))^2}{2}\leq 2-\lambda_N\leq 2(1-\overline{h}(1)).
\end{equation}

Our main result in this paper is the following higher-order dual Cheeger inequality.
\begin{thm}\label{TheoremMain}
For every graph $G$, and each natural number $1\leq k\leq N$, we have
\begin{equation}
\frac{2-\lambda_{N-k+1}}{2}\leq 1-\overline{h}(k)\leq Ck^3\sqrt{2-\lambda_{N-k+1}},
\end{equation}
or in another form,
\begin{equation}\label{My2}
\frac{1}{C^2k^6}(1-\overline{h}(k))^2\leq 2-\lambda_{N-k+1}\leq 2(1-\overline{h}(k)),
\end{equation}
where $C$ is a universal constant.
\end{thm}
This can be considered as a strong quantitative version  of the fact that $\lambda_{N-k+1}=2$ if and only if $G$ has at least $k$ bipartite connected components (see Proposition \ref{pro1} (i)).

Dually, when $k>\frac{N}{2}$, at least one of the subset pairs $\{(V_{2i-1}, V_{2i})\}_{i=1}^k\in \text{Pair}(k)$ has to contain an empty subset, hence $\overline{h}(k)=0$. Therefore (\ref{My2}) is more useful for the second half part of the spectrum.

\subsection {Clustering on real projective spaces}
The lower bound estimate of $2-\lambda_{N-k+1}$ in (\ref{My2}) is the essential part of Theorem \ref{TheoremMain}. For the proof, we will follow the route in Lee-Oveis Gharan-Trevisan \cite{LGT2013} which justifies the spectral clustering algorithms using the bottom $k$ eigenfunctions of \cite{NJW2002}. (Note that by this route, one can also only get an order $k^3$ in (\ref{LGT1}), Lee-Oveis Gharan-Trevisan used other strong techniques to derive $k^2$ for the price of a much larger $C$.) The novel point of our proof is to explore a new type of spectral clustering.

For an orthogonal system of eigenfunctions $f_1, f_2, \ldots, f_k$ of the normalized Laplace operator $\Delta$, one can construct the mapping
\begin{equation}\label{MapF}
 F:V\rightarrow \mathbb{R}^k, v\mapsto (f_1(v), f_2(v), \ldots, f_k(v)).
\end{equation}
For illustration, we ignore those vertices on which $F$ vanishes and consider the induced mapping to a unit sphere,
\begin{equation*}
 \widetilde{F}:V\rightarrow \mathbb{S}^{k-1}, v\mapsto \frac{F(v)}{\|F(v)\|},
\end{equation*}
where $\|\cdot\|$ is the Euclidean norm in $\mathbb{R}^k$. We will also use $\langle\cdot, \cdot\rangle$ for the inner product of vectors in $\mathbb{R}^k$.

The spectral clustering algorithms using the bottom $k$ eigenfunctions aim at obtaining $k$ subsets of $V$ with smaller expansions, i.e. clustering those groups of vertices which are closely connected inside the group and loosely connected with outside vertices. Roughly speaking, \cite{LGT2013} used the sphere distance to cluster vertices in $V$ via their image on the unit sphere under $\widetilde{F}$.

We explore the clustering phenomenon using the top $k$ eigenfunctions $f_{N-k+1}, \ldots, f_N$. Now use these functions in the definition of $F$ given in (\ref{MapF}). We first observe that for any $u\in V$
\begin{equation*}
 \frac{1}{d_u}\sum_{v, v\sim u}\langle F(u), F(v)\rangle w_{uv}=\sum_{j=N-k+1}^N(1-\lambda_j)f_j(u)^2\leq (1-\lambda_{N-k+1})\|F(u)\|^2.
\end{equation*}
Therefore if $\lambda_{N-K+1}>1$ is large, there exists at least one neighbor $v_0$ of $u$ such that $\langle F(u), F(v_0)\rangle<0$. That is, every vertex has always at least one neighbor far away from it under the sphere distance. This indicate that the aim of a proper clustering in this case should be different. In fact, instead of pursuing small expansion subsets, we aim at finding $k$ subsets, each of which has a bipartition such that the quantity $1-\overline{\phi}$ is small. Roughly speaking, we hope to find $k$ subsets whose induced subgraphs are all close to bipartite ones.

Let us explain how real projective spaces come into the situation by the following extremal but inspiring example. Consider a disconnected graph $G$ which has two bipartite connected components.
\begin{figure}[!htb]
\centering
\includegraphics[width=.8\textwidth]{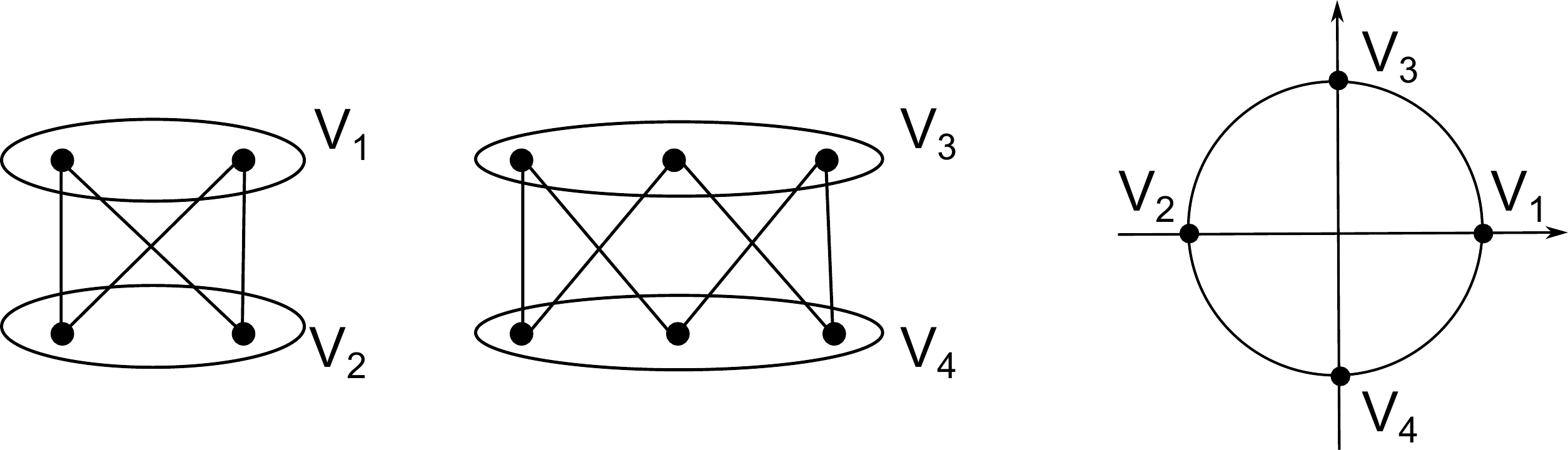}
\caption{The graph $G$ and its embedding into the sphere\label{figure1}}
\end{figure}
Then the embedding of its vertices into the sphere $\mathbb{S}^1$ via its top two eigenfunctions are shown in Figure \ref{figure1}.
If we use the sphere distance, we will obtain two clusters, e.g. $V_1\cup V_3$ and $V_2\cup V_4$. But we actually hope to find the clusters $V_1\cup V_2$ and $V_3\cup V_4$. A solution for this problem is to identify the antipodal points of the sphere and to obtain two clusters $V_1=V_2$ and $V_3=V_4$. Afterwards, we "unfold" each cluster to get two pairs of subsets which we desire. Therefore, we should use the metric on the real projective space instead of the sphere.

To understand the above clustering via top $k$ eigenfunctions more intuitively, we can think of the edges in $E$ as "hostile" relations. Vertices are clustered because they share common enemies. In contrast, the traditional clustering via the bottom $k$ eigenfunctions treat edges as "friendly" relations. We anticipate applications of this kind of hostile spectral clustering methods in practical fields, e.g. the research of social relationship networks. This hostile clustering is technically quite crucial for our purpose of proving Theorem \ref{TheoremMain}, as discussed in Lemmas \ref{lemmaspreading} and \ref{lemmalocalization}.

\subsection{Organization of the paper} In Section \ref{sectionPreliminaries} we collect necessary results from spectral graph theory and random partition theory of doubling metric spaces. Section \ref{sectionRelations} is devoted to various interesting relations between $\overline{h}(k)$ and $h(k)$. We discuss the lower bound estimates of $\lambda_{N-k+1}$ in Section \ref{sectionLowerbound}. In Section \ref{SectionMain} and \ref{sectionClustering} we present the proof of the lower bound estimate of $2-\lambda_{N-k+1}$ of (\ref{My2}). In Section \ref{sectionExamples}, we prove for cycles a slightly "shifted" version of higher-order dual Cheeger inequalities with an absolute constant which is even independent of $k$, based on the results of \cite{DJM2012}. We also analyze the example of unweighted cycles in detail. In Section \ref{Application}, we explore an application of higher-order dual Cheeger inequalities to the essential spectrum of a general reversible Markov operator.

We comment that the results about weighted graphs in this paper (except for Proposition \ref{pro2}) can be extended to graphs permitting self-loops, or in the language of Markov chains, lazy random walks. One just need to be careful about the fact $\mu(u)\geq \sum_{v,v\sim u, v\neq u}w_{uv}$ (see below for $\mu$) in that case.

\section{Preliminaries}\label{sectionPreliminaries}
\subsection{Spectral theory for normalized graph Laplacian}
We assign a natural measure $\mu$ to $V$ that $\mu(u)=d_u$, for every $u\in V$. The inner product of two functions $f,g: V \to \R$ is given by
$$ (f,g)_\mu = \sum_{u \in V} \mu(u)f(u)g(u). $$
We denote $l^2(V,\mu)$ the Hilbert space of functions on $V$ with the above inner product.

The normalized graph Laplacian $\Delta$ is
defined as follows. For any $f\in l^2(V,\mu)$, and $u\in V$
\begin{equation}\label{operator} \Delta f(u) := \frac{1}{d_u} \sum_{v, v \sim u} w_{uv} (f(u)-f(v)). \end{equation}
In matrix form, $\Delta=I-P$, where $I$ is the identity matrix, and $P:=D^{-1}A$, $D^{-1}f(u):=d_u^{-1}f(u)$ and $Af(u):=\sum_{v,v\sim u}w_{uv}f(v)$.

For a map $F: V\rightarrow \mathbb{R}^k$, we denote the Rayleigh quotient of $F$ by
\begin{equation}
\mathcal{R}(F):=\frac{\sum_{e=\{u,v\}\in E}\|F(u)-F(v)\|^2w_{uv}}
  {\sum_{u\in V}\|F(u)\|^2\mu(u)},
\end{equation}
and a dual version of the Rayleigh quotient of $F$
\begin{equation}
\overline{\mathcal{R}}(F):=\frac{\sum_{e=\{u,v\}\in E}\|F(u)+F(v)\|^2w_{uv}}
  {\sum_{u\in V}\|F(u)\|^2\mu(u)}.
\end{equation}
The support of a map $F$ is defined as
\begin{equation*}
\supp(F):=\{v\in V: F(v)\neq 0\}.
\end{equation*}

We call $\lambda$ an eigenvalue of $\Delta$ if there exists some $f\not\equiv 0$ with $\Delta f=\lambda f$.
Let $0=\lambda_1\leq \lambda_2\leq\cdots\leq\lambda_N$ be all the
eigenvalues of $\Delta$.
The Courant-Fischer-Weyl min-max principle tells us
\begin{equation}\label{keigenvalue}
\lambda_k=\min_{\substack{f_1,f_2,\ldots, f_k\not\equiv 0\\(f_i,f_j)_{\mu}=0, \forall i\neq j}}\max_{\substack{f\not\equiv 0\\f\in \text{span}\{f_1, f_2,\ldots, f_k\}}}
\mathcal{R}(f),
\end{equation}
and, dually,
\begin{equation}\label{keigenvalue2}
\lambda_{N-k+1}=\max_{\substack{f_1,f_2,\ldots, f_{k}\not\equiv 0\\(f_i,f_j)_{\mu}=0, \forall i\neq j}}\min_{\substack{f\not\equiv 0\\f\in \text{span}\{f_1, f_2,\ldots, f_{k}\}}}
\mathcal{R}(f).
\end{equation}

The next lemma can be found in Bauer-Jost \cite{BJ13} (Lemma 3.1 there). For its various variants, see e.g. \cite{Chung}, \cite{LGT2013}.
\begin{lemma}\label{CheegerRayleigh}
For any nonnegative function $g$ with $\supp(g)\neq \emptyset$,  there exist a subset $\emptyset\neq S\subseteq \supp(g)$ such that
\begin{equation}
\mathcal{R}(g)\geq 1-\sqrt{1-\phi(S)^2}.
\end{equation}
\end{lemma}

The next lemma is basically contained in the proof of Theorem 3.2 in Bauer-Jost \cite{BJ13}.
\begin{lemma}\label{Dual CheegerRayleigh}
For any function $f$ with $\supp(f)\neq \emptyset$, there exist two subsets $V_1, V_2\subseteq \supp(f)$ such that
$V_1\cap V_2=\emptyset$, $V_1\cup V_2\neq \emptyset$ and
\begin{equation}
\overline{\mathcal{R}}(f)\geq 1-\sqrt{1-(1-\overline{\phi}(V_1, V_2))^2}.
\end{equation}
\end{lemma}
We remark that here we do not require each of $V_1, V_2$ to be non-empty, but only their union. This lemma is derived from the combination of Lemma \ref{CheegerRayleigh} and a construction in Bauer-Jost \cite{BJ13} (following previous ideas in Desai-Rao \cite{DR1994}). For convenience, we recall the proof here briefly.
\begin{proof}
Denote $P(f):=\{v\in V: f(v)>0\}$, $N(f):=\{v\in V: f(v)<0\}$. By the assumption, $P(f)\cup N(f)\neq \emptyset$. Now construct a new graph $G'=(V', E', w')$ from the original graph $G$ in the following way. Duplicate all the vertices in $P(f)$ and $N(f)$. Denote by $u'$ the new vertices duplicated from $u$. For any edge $\{u,v\}$ such that $u,v\in P(f) \text{ or } u,v \in N(f)$, replace it by two new edges $\{u, v'\}, \{v, u'\}$ with the same weight $w'_{uv'}=w'_{vu'}=w_{uv}$. All the other vertices, edges and weights are unchanged.

Consider the function $g: V'\rightarrow \mathbb{R}$,
\begin{equation*}
g(u)=\left\{
  \begin{array}{ll}
    |f(u)|, & \hbox{if $u\in P(f)\cup N(f)$;} \\
    0, & \hbox{otherwise.}
  \end{array}
\right.
\end{equation*}
Then the above construction convert the inside edges of $P(f), N(f)$ into the boundary edges of $\supp(g)$. Furthermore, one can check that
\begin{equation*}
 \overline{\mathcal{R}}(f)\geq \mathcal{R}'(g).
\end{equation*}
Now by Lemma \ref{CheegerRayleigh}, we know there exists a subset $\emptyset\neq S\subseteq \supp(g)=P(f)\cup N(f)$, such that $\mathcal{R}'(g)\geq 1-\sqrt{1-\phi'(S)^2}$. Denote $S_P=S\cap P(f)$, $S_N=S\cap N(f)$. Then we have $S_P\cap S_N=\emptyset$, $S_P\cup S_N\neq \emptyset$ and
\begin{align*}
\phi'(S)&=\frac{|E'(S, \overline{S})|}{\text{vol}'(S)}=\frac{|E(S_P, S_P)|+|E(S_N, S_N)|+|E(S_P\cup S_N, \overline{S_P\cup S_N})|}{\text{vol}(S_P\cup S_N)}\\
&=1-\overline{\phi}(S_P, S_N),
\end{align*}
where for the last equality we used
\begin{align}
\text{vol}(S_P\cup S_N)=&2|E(S_P, S_N)|+|E(S_P, S_P)|+|E(S_N, S_N)|\notag
\\&+|E(S_P\cup S_N, \overline{S_P\cup S_N})|.\label{decomp of volumen}
\end{align}
This complete the proof of the lemma.
\end{proof}
%

\subsection{Padded random partitions of doubling metric space}
Random partition theory of metric spaces was firstly developed in theoretical computer science. It has found many important applications in pure mathematics, see e.g. \cite{LeeNaor2005}, \cite{KLPT2011}, \cite{LGT2013}. We discuss a result of that in this section which is needed in our arguments later.

We first introduce the concept of doubling metric spaces. There are two kinds of doubling properties: metric doubling and measure doubling.

The metric doubling constant $\rho_X$ of a metric space $(X,d)$ is defined as
\begin{align*}
 \rho_X:=\inf\{c\in \mathbb{N}:\,\,\, &\forall x\in X, r>0, \exists\,\, x_1, \ldots, x_c\in X,\\
  &\text{such that } B(x, r)\subseteq\bigcup_{i=1}^c B\left(x_i, \frac{r}{2}\right)\},
\end{align*}
where $B(x,r)$ is the closed ball in $X$ with center $x$ and radius $r$.
$(X,d)$ is called a metric doubling space if $\rho_X<+\infty$. The metric doubling dimension of $(X,d)$ is then defined as $\dim_d(X):=\log_2\rho_X$.

A Borel measure $\mu$ on $(X,d)$ is called a doubling measure if there exists a number $C_{\mu}$ such that for any $x\in X$, $r>0$,
\begin{equation*}
 0<\mu(B(x,r))\leq C_{\mu}\mu(B(x,\frac{r}{2}))< +\infty.
\end{equation*}
Similarly we call $\dim_{\mu}(X):=\log_2C_{\mu}$ the measure doubling dimension. Note that the measure doubling dimension of $\mathbb{R}^k$ with the standard Euclidean volume measure is exactly $k$.

The two doubling dimensions are related by the following result (see e.g. the Remark on p. 67 of \cite{CoiWei1971}).
\begin{lemma}\label{lemmaTwodimensions}
 If a metric space $(X,d)$ have a doubling measure $\mu$, then
\begin{equation}
 \dim_d(X)\leq 4\dim_{\mu}(X).
\end{equation}
\end{lemma}

A partition of a metric space $(X,d)$ is a series of subsets $P=\{S_i\}_{i=1}^{m}$ for some number $m$, where $S_i\cap S_j=\emptyset$, for any $i\neq j$ and $X=\bigcup_{i=1}^mS_i$. A partition can also be considered as a map $P: X\rightarrow 2^X$, such that $P(x)$ is the unique set in $\{S_i\}_{i=1}^m$ that contains $x$. A random partition $\mathcal{P}$ of $X$ is a distribution $\nu$ over the space of partitions of $X$. The following padded random partition theorem is a slightly modification of Theorem 3.2 in Gupta-Krauthgamer-Lee \cite{GKL2003}, (see also Lemma 3.11 in \cite{LeeNaor2005}).
\begin{thm}\label{thmPadderandompartition}
 Let $(X,d)$ be a finite metric subspace of $(Y, d)$. Then for every $r>0$, $\delta\in (0,1)$ there exists a random partition $\mathcal{P}$, i.e. a distribution $\nu$ over all possible partitions of $X$, such that
\begin{itemize}
  \item $\diam(S)\leq r$, for any $S$ in every partition $P$ in the support of $\nu$;
  \item $\mathbb{P}_{\nu}[B_d(x, \frac{r}{\alpha})\subseteq \mathcal{P}(x)]\geq 1-\delta$ for all $x$, where $\alpha=\frac{32\dim_d(Y)}{\delta}$.
\end{itemize}
\end{thm}
The random partition obtained in the above theorem is called a $(r, \alpha, 1-\delta)$-padded random partition in \cite{LGT2013}.
\begin{proof}
 We refer the readers to \cite{GKL2003} for the proof of this theorem. But we comment here that one can replace the $\dim_d(X)$ in Theorem 3.2 of \cite{GKL2003} by $\dim_d(Y)$ as we do here in the conclusion. The reason is that the only point where the metric doubling dimension plays a role in the proof is the following fact (this is more clear in \cite{LeeNaor2005}). Let $Z\subseteq X$ be a subset in which each pair of distinct points has a distance at least $\epsilon$. Then the cardinality of $B_d(x,t)\cap Z$ is less or equal to $2^{\dim_d(X)\left\lceil\log_2\frac{2t}{\epsilon}\right\rceil}$, for any $x\in X$ and radius $t\geq \epsilon$. Surely one can estimate the same cardinality in the bigger space $Y$.
\end{proof}

\section{Relations between $h(k)$ and its dual $\overline{h}(k)$}\label{sectionRelations}
In this section, we explore some interesting relations between the multi-way Cheeger and dual Cheeger constants. The following two propositions can be considered as strong extensions of Theorem 3.1 and Proposition 3.1 in Bauer-Jost \cite{BJ13}.
\begin{pro}\label{pro1}
 Let $G$ be any graph, then for each $1\leq k\leq N$ we have
\begin{equation}\label{connectionwithDual}
\overline{h}(k)\leq 1-h(k).
\end{equation}
Moreover, we have

(\textrm{i}) $\overline{h}(k)=1$ if and only if $G$ has at least $k$ connected components, each of which is bipartite.

(\textrm{ii}) If $G$ is bipartite, then $h(k)+\overline{h}(k)=1$, for each $1\leq k\leq N$.

(\textrm{iii}) If $h(k)+\overline{h}(k)=1$, then for the $k$-sub-bipartition $\{(V_{2i-1}, V_{2i})\}_{i=1}^k$ which assumes $\overline{h}(k)$, the pair $(V_{2i_0-1}, V_{2i_0})$ with the maximal expansion, i.e. $\phi(V_{2i_0-1}\cup V_{2i_0})=\max_{1\leq i\leq k}\phi(V_{2i-1}\cup V_{2i})$, is bipartite in the sense of $$|E(V_{2i_0-1}, V_{2i_0-1})|=|E(V_{2i_0}, V_{2i_0})|=0.$$
\end{pro}
\begin{proof}
By definition of $h(k)$, $\overline{h}(k)$ and the formula (\ref{decomp of volumen}), we have
\begin{align}
&1-\overline{h}(k)\notag\\
&=\min_{\{(V_{2i-1}, V_{2i})\}_{i=1}^k\in \text{Pair}(k)}\max_{1\leq i\leq k}\frac{\text{vol}(V_{2i-1}\cup V_{2i})-2|E(V_{2i-1}, V_{2i})|}{\text{vol}(V_{2i-1}\cup V_{2i})}\notag\\
&=\min_{\substack{\{S_i\}_{i=1}^k\\S_i\neq\emptyset, \forall i \\S_i\cap S_j=\emptyset, \forall i\neq j}}\min_{\substack{\{(V_{2i-1}, V_{2i})\}_{i=1}^k\\V_{2i-1}\cup V_{2i}=S_i\\V_{2i-1}\cap V_{2i}=\emptyset, \forall i}}\max_{1\leq i\leq k}\frac{|E(S_i, \overline{S_i})|+|E(V_{2i}, V_{2i})|+|E(V_{2i-1}, V_{2i-1})|}{\text{vol}(S_i)}\notag\\
&\geq \min_{\substack{\{S_i\}_{i=1}^k\\S_i\neq\emptyset, \forall i\\ S_i\cap S_j=\emptyset, \forall i\neq j}}\max_{1\leq i\leq k}\frac{|E(S_i, \overline{S_i})|}{\text{vol}(S_i)}=h(k).\label{abovecalculation}
\end{align}
Observe further in the above calculation, that the equality in (\ref{abovecalculation}) can be achieved when the graph $G$ is bipartite. Since then for each $S_i$, we can always find a bipartition $S_i=V_{2i-1}\cup V_{2i}$ such that $|E(V_{2i-1}, V_{2i-1})|=|E(V_{2i}, V_{2i})|=0$. This actually proves (\text{ii}).

For (\textrm{i}), if $\overline{h}(k)=1$, then there exists  $\{(V_{2i-1}, V_{2i})\}_{i=1}^k\in \text{Pair}(k)$ such that
\begin{equation*}
 \min_{1\leq i\leq k} \overline{\phi}(V_{2i-1}, V_{2i})=1.
\end{equation*}
Hence for each $i$,
\begin{equation*}
 1=\frac{2|E(V_{2i-1}, V_{2i})|}{\text{vol}(V_{2i-1}\cup V_{2i})}.
\end{equation*}
Then recalling (\ref{decomp of volumen}), we obtain
\begin{equation*}
 0=|E(V_{2i}, V_{2i})|=|E(V_{2i-1}, V_{2i-1})|=|E(V_{2i-1}\cup V_{2i-1}, \overline{V_{2i-1}\cup V_{2i-1}})|,
\end{equation*}
which implies $\{V_{2i-1}\cup V_{2i}\}_{i=1}^k$ are $k$ connected components, each of which is bipartite.
Conversely, if we know that $G$ has  $k$ connected components, each of which is bipartite, we can choose the $k$-sub-bipartition to be the bipartitions of those $k$ components. Then by definition, we have
\begin{equation*}
 \overline{h}(k)\geq \min_{1\leq i\leq k}\overline{\phi}(V_{2i-1}, V_{2i})=1.
\end{equation*}
Together with (\ref{connectionwithDual}), we know $ \overline{h}(k)=1$.

For (\textrm{iii}), suppose $\{(V_{2i-1}, V_{2i})\}_{i=1}^k\in \text{Pair}(k)$ assumes $\overline{h}(k)$, then by definition,
\begin{equation*}
 1=\overline{h}(k)+h(k)\leq \min_{1\leq i\leq k}\overline{\phi}(V_{2i-1}, V_{2i})+\max_{i=1}^k\phi(V_{2i-1}\cup V_{2i}).
\end{equation*}
Let $V_{2i_0-1}\cup V_{2i_0}$ attain the maximum in the above inequality. Then we have
\begin{align*}
 1&=\overline{h}(k)+h(k)\leq \overline{\phi}(V_{2i_0-1}, V_{2i_0})+\phi(V_{2i_0-1}\cup V_{2i_0})\\
&=\frac{2|E(V_{2i_0-1},V_{2i_0})|+|E(V_{2i-1}\cup V_{2i-1}, \overline{V_{2i-1}\cup V_{2i-1}})|}{\text{vol}(V_{2i_0-1}\cup V_{2i_0})}\leq 1.
\end{align*}
This implies $0=|E(V_{2i_0-1}, V_{2i_0-1})|=|E(V_{2i_0}, V_{2i_0})|$.
\end{proof}

\begin{remark}\label{rmkDualityBipartite}
 The property (\textrm{ii}) above shows the duality between $h(k)$ and $\overline{h}(k)$. Recalling the fact that bipartiteness is equivalent to whenever $\lambda$ is an eigenvalue, so is $2-\lambda$ (see e.g. Lemma 1.8 in \cite{Chung}), and employing property (\textrm{i}) above, we conclude
\begin{equation}\label{NiceDuality}
G \text{ is bipartite} \Leftrightarrow \lambda_k+\lambda_{N-k+1}=2, \,\forall\,k \Leftrightarrow h(k)+\overline{h}(k)=1, \,\forall\,  k.
\end{equation}

 We note that the equality $h(k)+\overline{h}(k)=1$ only for certain $k$ does not imply that $G$ is bipartite. For example, we have trivially for any graph with $N$ vertices, when $k>\frac{N}{2}$, we have
$$h(k)=1, \,\,\,\, \overline{h}(k)=0,$$
i.e. $h(k)+\overline{h}(k)=1$. We also have the following example.
\begin{example}
 Consider the unweighted (i.e. every edge has a weight $1$) complete graph $K_{2n}$ with $2n$ vertices. By choosing $n$ disjoint edges, it is not hard to check that
$$h(n)=\frac{2n-2}{2n-1}, \,\,\,\, \overline{h}(n)=\frac{1}{2n-1}.$$
Therefore we have $h(n)+\overline{h}(n)=1$.
\end{example}
In fact even when $h(k)+\overline{h}(k)=1$ for $2\leq k\leq N$, the graph still can be non-bipartite. An example is an odd cycle (see Proposition \ref{proOddcycle}). However, this graph is already very close to a bipartite graph.
\end{remark}
By (\ref{NiceDuality}), it is immediately to see that for bipartite graphs the classical Cheeger inequality (\ref{classicalCheeger}) and Theorem \ref{TheoremLGT} are equivalent to the following dual estimates.
\begin{coro}
Let $G$ be any bipartite graph. Then

(\textrm{i}) $2-\lambda_{N-1}\geq \frac{(1-\overline{h}(2))^2}{2}$.

(\textrm{ii}) $2-\lambda_{N-k+1}\geq \frac{1}{C^2k^4}(1-\overline{h}(k))^2$ for each natural number $1\leq k\leq N$, where $C$ is a universal number.
\end{coro}
It is interesting to note that the fact dual to Bauer-Jost's dual Cheeger inequality (\ref{BJdualCheeger}) is not the inequality (\ref{classicalCheeger}) but the identity $\lambda_1=0$ for bipartite graphs.

\begin{pro}\label{pro2}
 For any graph $G$, we have for each $1\leq k\leq N$
\begin{equation}
 \overline{h}(k)\geq \frac{1}{2}(1-h(k)).
\end{equation}
\end{pro}
To prove this proposition, we need the following lemma of Bauer-Jost \cite{BJ13} (see also Theorem 4.2 in \cite{BHJ2012}).
\begin{lemma}\label{lemmaBauerJost}
 For any subset $S\subseteq V$, there exists a partition $S=V_1\cup V_2$ such that
\begin{equation}
 |E(V_1, V_2)|\geq\max\{|E(V_1, V_1)|, |E(V_2, V_2)|\}.
\end{equation}
\end{lemma}
\begin{proof}[Proof of Proposition \ref{pro2}]
For any $k$-subpartition $S_1, S_2, \ldots, S_k$ of $V$, by Lemma \ref{lemmaBauerJost}, for each $1\leq i\leq k$ we have a partition $S_i=V_{2i-1}\cup V_{2i}$,
such that
\begin{equation}\label{ineqFromBJlemma}
 |E(V_{2i-1}, V_{2i})|\geq\max\{|E(V_{2i-1}, V_{2i-1})|, |E(V_{2i}, V_{2i})|\}.
\end{equation}
By definition, we know that
\begin{align*}
 &\overline{h}(k)\geq\min_{1\leq i\leq k}\overline{\phi}(V_{2i-1}, V_{2i})\\
&=\min_{1\leq i\leq k}\left(\frac{2|E(V_{2i-1}, V_{2i})|+\frac{1}{2}|E(V_{2i-1}\cup V_{2i},\overline{V_{2i-1}\cup V_{2i}})|}{\text{vol}(V_{2i-1}\cup V_{2i})}-\frac{1}{2}\phi(V_{2i-1}\cup V_{2i})\right).
\end{align*}
Combining (\ref{decomp of volumen}) and (\ref{ineqFromBJlemma}), we arrive at
\begin{equation*}
 \text{vol}(V_{2i-1}\cup V_{2i})\leq 4|E(V_{2i-1}, V_{2i})|+|E(V_{2i-1}\cup V_{2i},\overline{V_{2i-1}\cup V_{2i}})|.
\end{equation*}
Therefore, we obtain
\begin{equation*}
 \overline{h}(k)\geq \min_{1\leq i\leq k}\left(\frac{1}{2}-\frac{1}{2}\phi(S_i)\right)=\frac{1}{2}\left(1-\max_{1\leq i\leq k}\phi(S_i)\right).
\end{equation*}
Since $S_1, S_2, \ldots, S_k$ are chosen arbitrarily, we obtain
\begin{equation*}
  \overline{h}(k)\geq \frac{1}{2}(1-h(k)).
\end{equation*}
\end{proof}

\section{Lower bound estimate of $\lambda_{N-k+1}$}\label{sectionLowerbound}
In this section, we prove the lower bound estimate of $\lambda_{N-k+1}$. For any $k$-sub-bipartition $\{(V_{2i-1}, V_{2i})\}_{i=1}^k$, let us denote $V^*:=V\setminus \bigcup_{i=1}^k(V_{2i-1}\cup V_{2i})$.
Then for every $k\in \mathbb{N}$, we define a new constant which is greater or equal to $\overline{h}(k)$, namely
\begin{equation*}\label{UpperMultiway dual Cheeger constant}
\overline{h}^*(k)=\max_{\substack{\{(V_{2i-1}, V_{2i})\}_{i=1}^k\\\in \text{Pair}(k)}}\min_{1\leq i\leq k}\frac{2|E(V_{2i-1}, V_{2i})|+\frac{1}{2}| E(V_{2i-1}\cup V_{2i}, V^*)|}{\text{vol}(V_{2i-1}\cup V_{2i})}.
\end{equation*}

We prove the following result.
\begin{thm}\label{thmLowerbound}
 For any graph $G$ and each $1\leq k\leq N$, we have
\begin{equation}
 \lambda_{N-k+1}\geq 2\overline{h}^*(k).
\end{equation}
\end{thm}
Observe that the right hand side of (\ref{My2}) is an immediate corollary of this result.

\begin{proof}
 Given any $\{(V_{2i-1}, V_{2i})\}_{i=1}^k\in \text{Pair}(k)$, we choose $l^2(V, \mu)$-orthogonal functions as follows.
\begin{equation*}
 f_i(v)=\left\{
          \begin{array}{ll}
            1, & \hbox{if $v\in V_{2i-1}$;} \\
            -1, & \hbox{if $v\in V_{2i}$;} \\
            0, & \hbox{otherwise,}
          \end{array}
        \right.i=1,2,\ldots, k.
\end{equation*}
By construction, we know that every $f_i$ is not identically $0$. Then (\ref{keigenvalue2}) tells us that
\begin{equation}\label{eigenlow}
 \lambda_{N-k+1}\geq \min_{a_1, a_2, \ldots,a_k}\left\{\left.\frac{\sum_{e=(u,v)\in E}(f(u)-f(v))^2w_{uv}}{\sum_{u\in V}f(u)^2\mu(u)}\right| f=\sum_{i=1}^ka_if_i\right\},
\end{equation}
where the maximum is taken over all collections of $k$ constants, at least one of which is non-zero.
It is straightforward to see
\begin{equation}
\sum_{u\in V}f(u)^2\mu(u)=\sum_{u\in V}\sum_{i=1}^ka_i^2f_i(u)^2\mu(u)=\sum_{i=1}^ka_i^2\text{vol}(V_{2i-1}\cup V_{2i}).
\end{equation}
For the numerator of the quotient in (\ref{eigenlow}), we have
\begin{align*}
 &\sum_{e=(u,v)\in E}w_{uv}(f(u)-f(v))^2=\frac{1}{2}\sum_{u\in V}\sum_{v\in V}w_{uv}(f(u)-f(v))^2\\
=&\frac{1}{2}\sum_{i=1}^k\left(\sum_{u\in V_{2i-1}}\sum_{v\in V_{2i}}+\sum_{u\in V_{2i}}\sum_{v\in V_{2i-1}}\right)w_{uv}4a_i^2\\
&+\frac{1}{2}\sum_{i=1}^k\left(\sum_{u\in V_{2i-1}\cup V_{2i}}\sum_{v\in V^*}+\sum_{u\in V^*}\sum_{v\in V_{2i-1}\cup V_{2i}}\right)w_{uv}a_i^2\\
&+\frac{1}{2}\sum_{i=1}^k\sum_{u\in V_{2i-1}\cup V_{2i}}\sum_{l=1, l\neq i}^k\sum_{v\in V_{2l-1}\cup V_{2l}}w_{uv}(f(u)-f(v))^2\\
\geq&\sum_{i=1}^ka_i^2\left(4|E(V_{2i-1}, V_{2i})|+|E(V_{2i}\cup V_{2i-1}, V^*)|\right).
\end{align*}
Then we obtain
\begin{align*}
 \lambda_{N-k+1}&\geq\min_{a_1,a_2,\ldots,a_k}\frac{\sum_{i=1}^ka_i^2\left(4|E(V_{2i-1}, V_{2i})|+|E(V_{2i-1}\cup V_{2i}, V^*)|\right)}{\sum_{i=1}^ka_i^2\text{vol}(V_{2i-1}\cup V_{2i})}\\
&\geq 2\min_{1\leq i\leq k}\frac{2|E(V_{2i-1}, V_{2i})|+\frac{1}{2}|E(V_{2i-1}\cup V_{2i}, V^*)|}{\text{vol}(V_{2i-1}\cup V_{2i})}.
\end{align*}
Since $\{(V_{2i-1}, V_{2i})\}_{i=1}^k\in \text{Pair}(k)$ are chosen arbitrarily, we obtain
\begin{equation*}
 \lambda_{N-k+1}\geq 2\overline{h}^*(k).
\end{equation*}
\end{proof}

\section{The metric for clustering via top $k$ eigenfunctions}\label{SectionMain}
In this section, we start to prove the lower bound estimate of $2-\lambda_{N-k+1}$ in (\ref{My2}). Recall that the max-min problem in (\ref{keigenvalue2}) is solved by the corresponding eigenfunctions. Hence for the top $k$ eigenfunctions $f_{N+k-1}, \ldots, f_N$, we have
\begin{equation*}
 \lambda_{N-k+1}=\mathcal{R}(f_{N-k+1})=\min_{N-k+1\leq j\leq N}\mathcal{R}(f_j).
\end{equation*}
Therefore
\begin{equation*}
 \lambda_{N-k+1}\leq\frac{\sum_{j=N-k+1}^N\sum_{e=\{u,v\}\in E}(f_j(u)-f_j(v))^2w_{uv}}{\sum_{j=N-k+1}^N\sum_{u\in V}f_j(u)^2\mu(u)}.
\end{equation*}
Then it is straightforward to calculate
\begin{align}
 2-\lambda_{N-k+1}\geq & \frac{\sum_{j=N-k+1}^N\sum_{e=\{u,v\}\in E}(f_j(u)+f_j(v))^2w_{uv}}{\sum_{j=N-k+1}^N\sum_{u\in V}f_j(u)^2\mu(u)}\notag\\
=&\frac{\sum_{e=\{u,v\}\in E}\|F(u)+F(v)\|^2w_{uv}}{\sum_{u\in V}\|F(u)\|^2\mu(u)}=\overline{\mathcal{R}}(F),\label{2minusEigen}
\end{align}
where $F$ is the map from $V$ to $\mathbb{R}^k$ defined by $F(v)=(f_{N-k+1}(v), \ldots, f_N(v))$, i.e. in the way of (\ref{MapF}). One can also obtain the fact (\ref{2minusEigen}) by applying the min-max principle directly to the operator $I+P$.

Following the route in \cite{LGT2013} for dealing with Rayleigh quotient of the bottom $k$ eigenfunctions, we will localize $F$ to be $k$ disjointly supported maps $\{\Psi_i\}_{i=1}^k$ for which $\overline{\mathcal{R}}(\Psi_i)$ can be controlled from above by $\overline{\mathcal{R}}(F)$. Afterwards, we will apply Lemma \ref{Dual CheegerRayleigh} to handle each $\overline{\mathcal{R}}(\Psi_i)$ further. More explicitly, our requirements for the localization are for each $i$
\begin{itemize}
  \item $\sum_{u\in V}\|\Psi_i\|^2\mu(u)$ can be bounded from below by a certain fraction of $\sum_{u\in V}\|F(u)\|^2\mu(u)$;
  \item $\|\Psi_i(u)+\Psi_i(v)\|$ can be controlled from above by $\|F(u)+F(v)\|$.
\end{itemize}
The first requirement will be realized by the theory of random partitions on doubling metric spaces combined with the crucial Lemma \ref{lemmaspreading} below. The second requirement is solved by Lemma \ref{lemmalocalization} below. Before all those arguments, we first need to introduce our new metric.

\subsection{Real projective space with a rough metric}
We can use the standard Riemannian metric on real projective spaces inherited from the spheres via the canonical antipodal projection
$$Pr: \mathbb{S}^{k-1}\rightarrow P^{k-1}\mathbb{R}, \{x, -x\}\mapsto [x].$$
But for ease of calculations, we adopt a rough metric. That is, for any $[x],[y]\in P^{k-1}\mathbb{R}$, we define
\begin{equation}
  \overline{d}([x],[y]):=\min\{\|x+y\|, \|x-y\|\},
\end{equation}
where $\|\cdot\|$ is the Euclidean norm of vectors in $\mathbb{S}^{k-1}\subset \mathbb{R}^k$. It is easy to check that $\overline{d}$ is a metric on $P^{k-1}\mathbb{R}$.
\begin{pro}\label{proMetric}For the metric space $(P^{k-1}\mathbb{R}, \overline{d})$, we have

(\textrm{i}) $\diam (P^{k-1}\mathbb{R}, \overline{d})= \sqrt{2}$;

(\textrm{ii}) $\dim_{\overline{d}}(P^{k-1}\mathbb{R})\leq  4\left(\log_2\pi-\frac{1}{2}\right)(k-1)$.
\end{pro}
\begin{proof}
Let us denote the distance function deduced from the standard Riemannian metric on $P^{k-1}\mathbb{R}$ by $d_{\text{Rie}}$, and the Riemannian volume measure by $\mu_{\text{Rie}}$.
(\textrm{i}) is easy. (Compare the fact that $\diam(P^{k-1}\mathbb{R},d_{\text{Rie}})=\frac{\pi}{2}$). One can further observe that
\begin{equation}\label{LipschitzEquiv}
\frac{2\sqrt{2}}{\pi}d_{\text{Rie}}\leq \overline{d}\leq d_{\text{Rie}}.
\end{equation}

Since $(P^{k-1}\mathbb{R},d_{\text{Rie}})$ has constant sectional curvature $1$ (find more geometric properties of projective spaces in \cite{GHLRieGeo}), by the Bishop-Gromov comparison theorem,
\begin{equation*}
 \frac{\mu_{\text{Rie}}(B_{\overline{d}}(x,r))}{\mu_{\text{Rie}}(B_{\overline{d}}(x,\frac{r}{2}))}\leq \frac{\mu_{\text{Rie}}(B_{d_{\text{Rie}}}(x,\frac{\pi}{2\sqrt{2}}r))}{\mu_{\text{Rie}}(B_{d_{\text{Rie}}}(x,\frac{r}{2}))}\leq \left(\frac{\pi}{\sqrt{2}}\right)^{k-1}.
\end{equation*}
Furthermore, recalling Lemma \ref{lemmaTwodimensions}, this implies that
\begin{equation*}
\dim_{\overline{d}}(P^{k-1}\mathbb{R})\leq 4\log_2\left(\frac{\pi}{\sqrt{2}}\right)^{k-1}=4\left(\log_2\pi-\frac{1}{2}\right)(k-1).
\end{equation*}
\end{proof}

Consider the vertex set $V$ of a graph $G$, and a nontrivial map $F: V\rightarrow \mathbb{R}^k$. We write $\widetilde{V}_F:=\supp{F}$ for convenience. Then we define a map to the real projective space,
\begin{equation}
Pr\circ \widetilde{F}: \widetilde{V}_F\rightarrow P^{k-1}\mathbb{R}, v\mapsto Pr\left(\frac{F(v)}{\|F(v)\|}\right).
\end{equation}

Via the metric $\overline{d}$ defined above, we obtain a non-negative symmetric function on $\widetilde{V}_F\times \widetilde{V}_F$
\begin{align*}
  \overline{d}_F(u,v)&:=\overline{d}(Pr\circ \widetilde{F}(u),Pr\circ \widetilde{F}( v))\\
&=\min\left\{\left\|\frac{F(u)}{\|F(u)\|}-\frac{F(v)}{\|F(v)\|}\right\|, \left\|\frac{F(u)}{\|F(u)\|}+\frac{F(v)}{\|F(v)\|}\right\|\right\},
\end{align*}
which satisfies the triangle inequality on $\widetilde{V}_F$. That is, we obtain a pseudo metric space $(\widetilde{V}_F, \overline{d}_F)$.

\subsection{Spreading lemma}
We prove the following spreading lemma for the new metric extending Lemma 3.2 in Lee-Oveis Gahran-Trevisan \cite{LGT2013}.
For any map $F: S\subseteq V\rightarrow \mathbb{R}^k$,
let us call the quantity
$$\sum_{u\in S}\mu(u)\|F(u)\|^2$$
the $l^2$ mass of $F$ on $S$, denoted by $\mathcal{E}_S$ for short. By spreading, we mean that the $l^2$ mass of $F$ distributes evenly on $\widetilde{V}_F$.
\begin{lemma}\label{lemmaspreading}
Let $F$ be the map constructed from $l^2(V,\mu)$-orthonormal functions $f_1, f_2, \ldots, f_k$ in (\ref{MapF}). If $S\subseteq V$ satisfies $\diam(S\cap \widetilde{V}_F, \overline{d}_F)\leq r$, for some $0<r<1$,, then we have $$\mathcal{E}_S\leq \frac{1}{k(1-r^2)}\mathcal{E}_V.$$
\end{lemma}
In the above, we used the fact that $\mathcal{E}_S=\mathcal{E}_{S\cap\widetilde{V}_F}$, $\mathcal{E}_V=\mathcal{E}_{\widetilde{V}_{F}}$.
The map $F$ is said to be $(r, \frac{1}{k(1-r^2)})$-spreading if it satisfies the conclusion of the lemma. This property tells us that when the subset is of small size, the $l^2$ mass of $F$ of it can not be too large. The $l^2$ mass of $F$ cannot concentrate in a particular small region.

\begin{proof}
Since $\mathcal{E}_S=\mathcal{E}_{S\cap\widetilde{V}_F}$, we can suppose w.l.o.g. that $S\subseteq \widetilde{V}_F$.
As in \cite{LGT2013}, for a unit vector $x\in \mathbb{R}^k$, i.e., $\|x\|^2=\langle x,x\rangle=\sum_{i=1}^kx_i^2=1$, we have
\begin{align*}
 \sum_{v\in V}\mu(v)\langle x, F(v)\rangle^2&=\sum_{v\in V}\mu(v)\left(\sum_{i=1}^kx_if_i(v)\right)^2=\sum_{v\in V}\sum_{i, j=1}^kx_ix_jf_i(v)f_j(v)\mu(v)\\&=\sum_{i,j=1}^kx_ix_j\sum_{v\in V}f_i(v)f_j(v)\mu(v)
=\sum_{i=1}^kx_i^2=1.
\end{align*}
We also have
\begin{equation*}
 \mathcal{E}_V=\sum_{v\in V}\sum_{i=1}^k\mu(v)f_i^2(v)=\sum_{i=1}^k\sum_{v\in V}\mu(v)f_i^2(v)=k.
\end{equation*}
Now, for any $u\in S$, we obtain
\begin{align}\label{spreadingCalculation}
 \frac{\mathcal{E}_V}{k}=1&=\sum_{v\in V}\mu(v)\left\langle F(v), \frac{F(u)}{\|F(u)\|}\right\rangle^2=\sum_{v\in V}\mu(v)\|F(v)\|^2\left\langle \frac{F(v)}{\|F(v)\|},\frac{F(u)}{\|F(u)\|}\right\rangle^2.
\end{align}
Note that
\begin{align*}
& \left\langle \frac{F(v)}{\|F(v)\|},\frac{F(u)}{\|F(u)\|}\right\rangle^2=\left\langle \frac{F(v)}{\|F(v)\|},\frac{-F(u)}{\|F(u)\|}\right\rangle^2\\
=&\left[\frac{1}{2}\left(2-\left\|\frac{F(v)}{\|F(v)\|}-\frac{F(u)}{\|F(u)\|}\right\|^2\right)\right]^2=\left[\frac{1}{2}\left(2-\left\|\frac{F(v)}{\|F(v)\|}+\frac{F(u)}{\|F(u)\|}\right\|^2\right)\right]^2.
\end{align*}
Therefore we have
\begin{align*}
 \left\langle \frac{F(v)}{\|F(v)\|},\frac{F(u)}{\|F(u)\|}\right\rangle^2=\left[1-\frac{1}{2}\overline{d}_F(u,v)^2\right]^2.
\end{align*}
Inserting this back into (\ref{spreadingCalculation}), we arrive at
\begin{equation*}
  \frac{\mathcal{E}_V}{k}\geq \sum_{v\in S}\mu(v)\|F(v)\|^2\left[1-\frac{1}{2}\overline{d}_F(u,v)^2\right]^2\geq \left(1-r^2+\frac{r^4}{4}\right)\mathcal{E}_S\geq (1-r^2)\mathcal{E}_S.
\end{equation*}
This proves the lemma.
\end{proof}

\subsection{Localization lemma}
Let $F$ be a map from $V$ to $\mathbb{R}^k$.
 Given $\epsilon>0$, we define the $\epsilon$-neighborhood of a subset $\widetilde{S}\subseteq \widetilde{V}_F$ with respect to the metric $\overline{d}_F$ as
\begin{equation}
 N_{\epsilon}(\widetilde{S}, \overline{d}_F):=\{v\in \widetilde{V}_F: \overline{d}_F(v, \widetilde{S})<\epsilon\}.
\end{equation}

Now for any given subset $S\subseteq V$, we define a cut-off function,
\begin{equation}\label{cutofffunction}
 \theta(v):=\left\{
              \begin{array}{ll}
                0, & \hbox{if $F(v)=0$;} \\
                \max\left\{0, 1-\frac{\overline{d}_F(v, S\cap \widetilde{V}_F)}{\epsilon}\right\}, & \hbox{otherwise.}
              \end{array}
            \right.
\end{equation}
Then we can localize $F$ to be
\begin{equation*}
 \Psi:=\theta\cdot F: V\rightarrow \mathbb{R}^k.
\end{equation*}
It is obvious that $\Psi\mid_{S}=F\mid_S$ and $\supp(\Psi)\subseteq N_{\epsilon}(S\cap \widetilde{V}_F, \overline{d}_F)$.
We extend Lemma 3.3 in \cite{LGT2013} to our new metric in the following localization lemma.
\begin{lemma}\label{lemmalocalization}
 Given $\epsilon<2$, let $\Psi$ be the localization of $F$ via $\theta$ as above. Then for any $e=\{u,v\}\in E$, we have
\begin{equation}\label{eqlocalization}
 \|\Psi(u)+\Psi(v)\|\leq\left(1+\frac{2}{\epsilon}\right)\|F(u)+F(v)\|.
\end{equation}
\end{lemma}
\begin{proof}
First observe that if $F(u)=F(v)=0$, (\ref{eqlocalization}) is trivial. If only one of $F(u), F(v)$ vanishes, (\ref{eqlocalization}) is implied by the fact that $|\theta|\leq 1$. Therefore we only need to consider the case that both $u, v\in \widetilde{V}_F$.

It is direct to calculate for $\{u,v\}\in E$
\begin{align}
& \|\Psi(u)+\Psi(v)\|=\|\theta(u)F(u)+\theta(v)F(v)\|\notag\\
\leq &|\theta(u)|\|F(u)+F(v)\|+|\theta(u)-\theta(v)|\|F(v)\|.\label{beginningoflocal}
\end{align}
If $\{u,v\}$ satisfies $\overline{d}_F(u,v)=\left\|\frac{F(u)}{\|F(u)\|}-\frac{F(v)}{\|F(v)\|}\right\|$, then by Proposition \ref{proMetric} (\text{i}) we know
\begin{equation*}
 \left\|\frac{F(u)}{\|F(u)\|}-\frac{F(v)}{\|F(v)\|}\right\|\leq \sqrt{2},
\end{equation*}
which implies $\langle F(u), F(v)\rangle\geq 0$. Therefore
\begin{equation*}
 |\theta(u)-\theta(v)|\|F(v)\|\leq \|F(v)\|\leq\|F(u)+F(v)\|.
\end{equation*}
If $\{u,v\}$ satisfies $\overline{d}_F(u,v)=\left\|\frac{F(u)}{\|F(u)\|}+\frac{F(v)}{\|F(v)\|}\right\|$, then $\langle F(u), F(v)\rangle\leq 0$, and
we have,
\begin{align*}
&|\theta(u)-\theta(v)|\|F(v)\|\leq\frac{1}{\epsilon}\overline{d}_F(u,v)\|F(v)\|=\frac{1}{\epsilon}\left\|\frac{\|F(v)\|}{\|F(u)\|}F(u)+F(v)\right\|\\
\leq&\frac{1}{\epsilon}\left(\|F(v)+F(u)\|+\left\|\frac{\|F(v)\|}{\|F(u)\|}F(u)-F(u)\right\|\right)\\
\leq &\frac{2}{\epsilon}\|F(u)+F(v)\|.
\end{align*}
In conclusion, we have
\begin{equation}
 |\theta(u)-\theta(v)|\|F(v)\|\leq \frac{2}{\epsilon}\|F(u)+F(v)\|.
\end{equation}
Recalling (\ref{beginningoflocal}) and the fact $|\theta|\leq 1$, we obtain inequality (\ref{eqlocalization}), finishing the proof of the lemma.
\end{proof}

\section{Finding $k$-sub-bipartition with small $1-\overline{\phi}$}\label{sectionClustering}
In this section, we will prove
\begin{thm}\label{thmUpperbound}
 For a map $F: V\rightarrow \mathbb{R}^k$ constructed from $l^2(V,\mu)$-orthonormal functions $f_1, f_2, \ldots, f_k$ as in (\ref{MapF}), there exists a $k$-sub-bipartition $\{(V_{2i-1}, V_{2i})\}_{i=1}^k$, such that for each $i$
\begin{equation}
 \frac{(1-\overline{\phi}(V_{2i-1}, V_{2i}))^2}{2}\leq 1-\sqrt{1-(1-\overline{\phi}(V_{2i-1}, V_{2i}))^2}\leq Ck^6\overline{\mathcal{R}}(F),
\end{equation}
where $C$ is a universal constant.
\end{thm}

 It was shown in (\ref{2minusEigen}) that once we take the $k$ functions above to be the top $k$ orthonormal eigenfunctions, we will have $2-\lambda_{N-k+1}\geq \overline{\mathcal{R}}(F)$. Then by the definition of $\overline{h}(k)$, (\ref{My2}) follows immediately from this theorem.

 We need the following lemma (modified from Lemma 3.5 of \cite{LGT2013}) to get $k$-disjoint subsets of $(\widetilde{V}_F, \overline{d}_F)$.
\begin{lemma}\label{lemmaFindsubsets}
 Suppose $F$ is $(r, \frac{1}{k}\left(1+\frac{1}{8k}\right))$-spreading, and $(\widetilde{V}_F, \overline{d}_F)$ has a $(r, \alpha, 1-\frac{1}{4k})$-padded random partition, then there exists $k$ non-empty, mutually disjoint subsets $T_1, T_2, \ldots, T_k\subseteq \widetilde{V}_F$ such that\begin{itemize}
      \item for any $1\leq i\neq j\leq k$, $\overline{d}_F(T_i, T_j)\geq 2\frac{r}{\alpha}$;
      \item for any $1\leq i\leq k$, $\mathcal{E}_{T_i}\geq \frac{1}{2k}\mathcal{E}_V$.
    \end{itemize}
\end{lemma}
\begin{proof}
Let $\mathcal{P}$ be the $(r, \alpha, 1-\frac{1}{4k})$-padded random partition in the assumption.
Denote by $I_{B_{\overline{d}_F}(v,\frac{r}{\alpha})\subseteq \mathcal{P}(v)}$ the indicator function for the event that $B_{\overline{d}_F}(v,\frac{r}{\alpha})\subseteq \mathcal{P}(v)$ happens. Then we calculate the expectation
\begin{align*}
 \mathbb{E}_{\mathcal{P}}\left[\sum_{v\in V}\mu(v)\|F(v)\|^2I_{B_{\overline{d}_F}(v,\frac{r}{\alpha})\subseteq \mathcal{P}(v)}\right]\geq \sum_{v\in V}\mu(v)\|F(v)\|^2\left(1-\frac{1}{4k}\right).
\end{align*}
If we denote $\widehat{S}:=\{v\in S: B_{\overline{d}_F}(v,\frac{r}{\alpha})\subseteq S\}$, this is equivalently to
\begin{align*}
\sum_{P\in\mathcal{P}}\left(\sum_{S\in P}\sum_{v\in\widehat{S}}\mu(v)\|F(v)\|^2\right)\mathbb{P}(P)\geq\left(1-\frac{1}{4k}\right)\mathcal{E}_V.
\end{align*}
Therefore there exist at least one partition $P=\{S_i\}_{i=1}^m$ of $\widetilde{V}_F$ for some natural number $m$ such that
\begin{equation}\label{totallowerbound}
 \sum_{i=1}^m\mathcal{E}_{\widehat{S}_i}\geq\left(1-\frac{1}{4k}\right)\mathcal{E}_V.
\end{equation}
By the spreading property in the assumption, we know for every $1\leq i\leq m$
\begin{equation}\label{singleupperbound}
 \mathcal{E}_{\widehat{S}_i}\leq\mathcal{E}_{S_i}\leq \frac{1}{k}\left(1+\frac{1}{8k}\right)\mathcal{E}_V.
\end{equation}

We can construct the desired $k$ disjoint subsets from $\widehat{S}_1, \widehat{S}_2, \ldots, \widehat{S}_m$ by the following procedure.
If we can find two of these sets, say $\widehat{S}_i, \widehat{S}_j$, such that
\begin{equation}\label{small two}
 \mathcal{E}_{\widehat{S}_i}<\frac{1}{2k}\mathcal{E}_V,  \mathcal{E}_{\widehat{S}_j}<\frac{1}{2k}\mathcal{E}_V,
\end{equation}
then we replace them by the set $\widehat{S}_i\cup \widehat{S}_j$. Note that in this process we did not violate the fact (\ref{singleupperbound}) since
\begin{equation}
 \mathcal{E}_{\widehat{S}_i\cup \widehat{S}_j}<\frac{1}{k}\mathcal{E}_V<\frac{1}{k}\left(1+\frac{1}{8k}\right)\mathcal{E}_V.
\end{equation}
We repeat the above operation until we can not find two sets anymore such that (\ref{small two}) holds. Therefore, when we stop, we get a series of subsets
$T_1, T_2, \ldots, T_r$ for some number $r$
such that
\begin{equation}\label{singleupperbound2}
 \mathcal{E}_{T_i}\leq \frac{1}{k}\left(1+\frac{1}{8k}\right)\mathcal{E}_V, \text{ for } i=1,2,\ldots,r,
\end{equation}
and
\begin{equation}
 \mathcal{E}_{T_i}\geq \frac{1}{2k}\mathcal{E}_V, \text{ for }i=1,2,\ldots, r-1.
\end{equation}
We are not sure about the lower bound of at most one of those $\{\mathcal{E}_{T_i}\}_{i=1}^r$, here we suppose w.l.o.g. that it is $\mathcal{E}_{T_r}$.

Observe that
\begin{equation}
 (k-1)\cdot\frac{1}{k}\left(1+\frac{1}{8k}\right)\leq 1-\frac{1}{4k}-\frac{1}{2k}.
\end{equation}
Recalling (\ref{totallowerbound}) and (\ref{singleupperbound2}), we know $r\geq k$, and if we take
\begin{equation*}
 T_k:=\bigcup_{j=k}^{r}T_j,
\end{equation*}
we will have \begin{equation*}
 \mathcal{E}_{T_k}\geq\frac{1}{2k}\mathcal{E}_V.
\end{equation*}
This proves the lemma.
\end{proof}
\begin{proof}[Proof of Theorem \ref{thmUpperbound}]
 Choosing $r=\frac{1}{3\sqrt{k}}$, by Lemma \ref{lemmaspreading} we have that $F$ is $(r, \frac{1}{k}\left(1+\frac{1}{8k}\right))$-spreading. If we further take $\delta=\frac{1}{4k}$, by
 Theorem \ref{thmPadderandompartition}, $V$ has a $(r, \alpha, 1-\frac{1}{4k})$-padded random partition with
$$\alpha=128k\dim_{\overline{d}}(P^{k-1}\mathbb{R}).$$
Recalling Proposition \ref{proMetric} (\textrm{ii}), we know that there exists an absolute constant $C=4(\log_2\pi-\frac{1}{2})$ such that $\alpha\leq 128Ck(k-1)$.
Then we can apply Lemma \ref{lemmaFindsubsets} to find $k$ disjoint subsets $T_1, T_2, \ldots, T_k$, such that
\begin{itemize}
  \item for any $1\leq i\neq j\leq k$, $\overline{d}_F(T_i, T_j)\geq 2\frac{r}{\alpha}\geq \frac{2}{3\sqrt{k}}\frac{1}{128Ck(k-1)}$;
  \item for any $1\leq i\leq k$, $\mathcal{E}_{T_i}\geq\frac{1}{2k}\mathcal{E}_V$.
\end{itemize}
Now let $\{\theta_i\}_{i=1}^k$, $k\geq 2$ be the $k$ cut-off functions defined in (\ref{cutofffunction}) (replacing $S$ there by $T_i$) with $\epsilon=\frac{1}{3\sqrt{k}}\frac{1}{128Ck(k-1)}$.
Then we get $k$ localizations $\Psi_i=\theta_iF$ of $F$ with disjoint supports. Note that $\Psi_i\mid_{T_i}=F\mid_{T_i}$. Applying Lemma \ref{lemmalocalization}, we arrive at
\begin{align}\label{final1}
 \overline{\mathcal{R}}(\Psi_i)\leq 2k\left(1+768C\sqrt{k}k(k-1)\right)^2\overline{\mathcal{R}}(F)\leq 2\times (768C)^2k^6\overline{\mathcal{R}}(F).
\end{align}
Let us write $\Psi_i(u)=(\Psi_i^1(u), \Psi_i^2(u),\ldots, \Psi_i^k(u))$. Then we can conclude that there exists a $j_0\in \{1,2,\ldots, k\}$ such that $\Psi_i^{j_0}$ is not identically zero and
\begin{equation}\label{final2}
 \overline{\mathcal{R}}(\Psi_i^{j_0})\leq \overline{\mathcal{R}}(\Psi_i).
\end{equation}
Then by Lemma \ref{Dual CheegerRayleigh}, for each $T_i$, we can find two subsets $V_{2i-1}, V_{2i}\subseteq T_i$ which satisfy $V_{2i-1}\cap V_{2i}=\emptyset$, $V_{2i-1}\cup V_{2i}\neq \emptyset$ such that
\begin{equation}\label{final3}
 1-\sqrt{1-(1-\overline{\phi}(V_{2i-1}, V_{2i}))^2}\leq \overline{\mathcal{R}}(\Psi_i^{j_0}).
\end{equation}
Combining (\ref{final1}), (\ref{final2}) and (\ref{final3}), we prove the theorem.
\end{proof}

Applying (\ref{final1}) and (\ref{final2}) to the top $k$ orthonormal eigenfunctions, we arrive at the following lemma.
\begin{lemma}
For every graph $G$, and each natural number $1\leq k\leq N$, there exist $k$ disjointly supported functions $\psi_1, \ldots, \psi_k: V\rightarrow \mathbb{R}$ such that for each $1\leq i\leq k$,
\begin{equation}
\overline{\mathcal{R}}(\psi_i)\leq Ck^6(2-\lambda_{N-k+1}),
\end{equation}
where $C$ is a universal constant.
\end{lemma}
Note in the above lemma that the case $k=1$ is trivial. Combining this lemma with Theorem 4.10 in \cite{KLLGT2013}, we prove the following \emph{improved}
version of higher order dual Cheeger inequalities.
\begin{thm}
For every graph $G$ and $1\leq k\leq l\leq N$, we have
\begin{equation}\label{improvedHOdualCheeger}
1-\overline{h}(k)\leq Clk^6\frac{2-\lambda_{N-k+1}}{\sqrt{2-\lambda_{N-l+1}}},
\end{equation}where $C$ is a universal constant.
\end{thm}
(\ref{improvedHOdualCheeger}) can be seen as a dual result to Corollary 1.3 (i) in \cite{KLLGT2013}.

\section{Trees and cycles}\label{sectionExamples}
In this section, we explore the relations between spectra and $k$-way Cheeger and dual Cheeger constants on trees (i.e. graphs without cycles) and cycles. In particular, we discuss the following kind of inequalities,
\begin{equation}\label{goals}
 \lambda_k\geq C_1(k)h(k)^2\text{  and  } 2-\lambda_{N-k+1}\geq C_2(k) (1-\overline{h}(k))^2.
\end{equation}
It is proved by Miclo \cite{Miclo2008} and Daneshgar-Javadi-Miclo \cite{DJM2012} that for trees, $C_1(k)$ can be $\frac{1}{2}$, and for cycles it can be
\begin{equation}\label{DJMconstants}
C_1(k)=\left\{
           \begin{array}{ll}
             \frac{1}{2}, & \hbox{if $k=1$ or $k$ is even;} \\
             \frac{1}{48}, & \hbox{if $k\geq 3$ is odd.}
           \end{array}
         \right.
\end{equation}
That is, for those special classes of graphs, $C_1(k)$ is even independent of $k$. (Of course the case $k=1$ is trivial and we list it here just for completeness.) Recalling (\ref{NiceDuality}) in Remark \ref{rmkDualityBipartite}, we can take $C_2(k)=C_1(k)$ to be independent of $k$ for even cycles and trees since they are all bipartite.

If we replace $\overline{h}(k)$ in (\ref{goals}) by $\overline{h}(k-1)$, we can prove the following result for cycles.

\begin{thm}\label{thmCycles}
For any cycle $\mathcal{C}_N$, we have for every $1\leq k\leq N$
\begin{equation}
2-\lambda_{N-k+1}\geq C_3(1-\overline{h}(k-1))^2,
\end{equation}
with the notation $\overline{h}(0):=\overline{h}(1)$ and $$C_3=\left\{
            \begin{array}{ll}
              \frac{1}{2}, & \hbox{if $N-k+1$ is odd or $N-k+1=N-1, N$;} \\
              \frac{1}{48}, & \hbox{if $N-k+1\leq N-2$ is even.}
            \end{array}
          \right.
$$
\end{thm}
As commented above, we only need to prove the theorem for odd cycles, for which we need the following observation.
\begin{pro}\label{proOddcycle}
For odd cycles, we have
\begin{equation}
 h(k)+\overline{h}(k)=1, \text{ for }2\leq k\leq N.
\end{equation}
\end{pro}
\begin{proof}
 Observe the fact that any proper subset $S$ of $V$ possesses a bipartition $V_1\cup V_2=S$ such that $|E(V_1, V_1)|=|E(V_2, V_2)|=0$. Therefore, the inequality (\ref{abovecalculation}) in the proof of  Proposition \ref{pro1} (\textrm{ii}) is in fact an equality, proving $h(k)+\overline{h}(k)=1$ for any $2\leq k\leq N$.
\end{proof}
\begin{proof}[Proof of Theorem \ref{thmCycles}]
Let $\mathcal{C}_{N-1}$ be a cycle obtained from $\mathcal{C}_N$ by contracting one edge $\{u_0,v_0\}$, i.e., by removing $\{u_0, v_0\}$ from the edge set of $\mathcal{C}_N$ and identifying vertices $u_0, v_0$ to be one vertex renamed as $\eta$ in $\mathcal{C}_{N-1}$. Let us use a prime to indicate quantities of the new graph $\mathcal{C}_{N-1}$. Then we have $d'_{\eta}=d_{u_0}+d_{v_0}-2w_{u_0v_0}$ and for $u_1\sim u_0\sim v_0\sim v_1$ in $\mathcal{C}_{N}$, $w'_{\eta u_1}=w_{u_0u_1}$, $w_{\eta v_1}=w_{v_0 v_1}$.

Now by an interlacing idea in Butler \cite{Butler2007}, we claim \begin{equation}
 \lambda'_k\geq  \lambda_k, \text{ for }1\leq k\leq N-1.
\end{equation}
(Note that the interlacing results for the so-called weak coverings in \cite{Butler2007} do not apply to our case since different weights and vertex degrees of the new graph $\mathcal{C}_{N-1}$ are chosen there. See also the general result Theorem 4.1 for contracting operations in Horak-Jost \cite{HJ2013} which works for unweighted graphs.) Indeed, by (\ref{keigenvalue2}) we have
\begin{equation*}
 \lambda'_k=\max_{\mathcal{F}^{N-k}\subseteq \mathbb{R}^{N-1}}\min_{f'\in \mathcal{F}^{N-k}, f'\not\equiv 0}\mathcal{R}'(f'),
\end{equation*}
where the maximum is taken over all possible $N-k$ dimensional subspaces $\mathcal{F}^{N-k}$ of $\mathbb{R}^{N-1}$. Given $f'$, we define a new function $f$ on $\mathcal{C}_N$ as
\begin{equation*}
 f(u)=\left\{
        \begin{array}{ll}
          f'(\eta), & \hbox{if $u=u_0$ or $v_0$;} \\
          f'(u), & \hbox{otherwise.}
        \end{array}
      \right.
\end{equation*}
This satisfies
\begin{equation*}
 \mathcal{R}'(f')=\frac{\sum_{u\sim v}w_{uv}(f(u)-f(v))^2}{\sum_u d_uf^2(u)-2w_{u_0v_0}(f'(\eta))^2}\geq \mathcal{R}(f).
\end{equation*}
Hence we have
\begin{align*}
\lambda'_k&= \max_{\mathcal{F}^{N-k+1}\subseteq \mathbb{R}^{N}}\min_{\substack{f\in \mathcal{F}^{N-k+1}, f\not\equiv 0\\f(u_0)=f(v_0)}}\frac{\sum_{u\sim v}w_{uv}(f(u)-f(v))^2}{\sum_u d_uf^2(u)-2w_{u_0v_0}(f(u_0))^2}\\&\geq \max_{\mathcal{F}^{N-k+1}\subseteq \mathbb{R}^{N}}\min_{f\in \mathcal{F}^{N-k+1}, f\not\equiv 0}\mathcal{R}(f)=\lambda_k.
\end{align*}
Using this and the bipartiteness of $\mathcal{C}_{N-1}$ for odd $N$, we obtain for $k\geq 2$,
\begin{equation}
 2-\lambda_{N-k+1}\geq 2-\lambda'_{N-k+1}=2-\lambda'_{(N-1)-(k-1)+1}=\lambda'_{k-1}\geq \lambda_{k-1}.
\end{equation}
By the result (\ref{DJMconstants}) of \cite{DJM2012} and Proposition \ref{proOddcycle}, we prove the theorem for odd $N$ and $k\geq 3$. For $k=1,2$, by the dual Cheeger inequality (\ref{BJdualCheeger}), $2-\lambda_{N-1}\geq 2-\lambda_{N}\geq \frac{1}{2}(1-\overline{h}(1))^2$.
\end{proof}

In the following, we consider the special class of unweighted cycles as an example. We will see that the constants in (\ref{DJMconstants}) can be better for unweighted cycles and $C_2(k)$ can also be independent of $k$ for unweighted odd cycles.

\begin{pro}
For an unweighted cycle, we have $h(1)=0$ and
\begin{equation*}
\overline{h}(1)=\left\{
                                                 \begin{array}{ll}
                                                  \frac{N-1}{N}, & \hbox{if $N$ is odd;} \\
                                                   1, & \hbox{if $N$ is even,}
                                                 \end{array}
                                               \right.
h(k)=1-\overline{h}(k)=\frac{1}{\left\lfloor\frac{N}{k}\right\rfloor}, \text{ for }2\leq k\leq N.
\end{equation*}
\end{pro}
\begin{proof}
For $2\leq k\leq N$, we only need to calculate $h(k)$ since we always have $h(k)=1-\overline{h}(k)$.

Let $\{S_i\}_{i=1}^k$ be the $k$-subpartition of a cycle achieving $h(k)$. Then we can always suppose that $V^*=V\setminus\bigcup_{i=1}^kS_i=\emptyset$ and every $S_i$ is connected, since otherwise we can construct $k$ connected partitions from them without increasing their expansions as follows.
For every $i$, if $S_i$ has connected components $\{S_i^j\}_{j=1}^t$ for some natural number $t\geq 2$, we replace $S_i$ by $S_i^{j_1}\in \{S_i^j\}_{j=1}^t$ which has the minimal expansion. Note that
\begin{equation*}
 \phi(S_{i}^{j_1})=\min_{1\leq j\leq t}\phi(S_{i}^j)\leq\frac{\sum_{j=1}^t|E(S_{i}^j, \overline{S_{i}^j})|}{\sum_{j=1}^t\text{vol}(S_{i}^j)}=\phi(S_{i}).
\end{equation*}
For the new $k$-subpartition $\{S_i\}_{i=1}^k$, if $V^*\neq \emptyset$, we can combine each connected component of $V^*$ with one of its adjacent subsets in $\{S_i\}_{i=1}^k$. Since the boundary measure is unchanged while the volume increases, we again do not increase their expansion in this process.

%
%

In an unweighted cycle for any connected non-empty proper subsets $S_i\subseteq V$, we have $\phi(S_i)=\frac{2}{2\sharp S_i}$, where $\sharp S_i$ represents the number of vertices in $S_i$. Then it is straightforward to obtain
\begin{equation*}
 h(k)=\min_{S_1, \ldots, S_k}\max_{1\leq i\leq k}\frac{1}{\sharp S_i}=\frac{1}{\left\lfloor\frac{N}{k}\right\rfloor}.
\end{equation*}
By similar arguments for odd $N$, we can prove $\overline{h}(1)=\frac{N-1}{N}$.
\end{proof}

It is known that the eigenvalues of an unweighted cycle (see e.g. Example 1.5 in \cite{Chung}), listed in an increasing order, are  $$\lambda_k=1-\cos \left(\frac{2\pi}{N}\left\lfloor\frac{k}{2}\right\rfloor\right), k=1,2,\ldots, N.$$ It is then straightforward to check that
\begin{equation*}
 2-\lambda_{N-k+1}=\left\{
             \begin{array}{ll}
               1-\cos\frac{(k-1)\pi}{N}, & \hbox{if $N-k+1$ is even;} \\
               1-\cos\frac{k\pi}{N}, & \hbox{if $N-k+1$ is odd.}
             \end{array}
           \right.
\end{equation*}
\begin{pro}\label{proUnweighted cycle}
For an unweighted cycle, we have for every $1\leq k\leq N$,
\begin{align}
 C_1h(k)^2\leq &\lambda_k\leq\frac{\pi^2}{2}h(k)^2, \label{exampleHCI}\\
C_2(1-\overline{h}(k))^2\leq 2-&\lambda_{N-k+1}\leq\frac{\pi^2}{2}(1-\overline{h}(k))^2,\label{exampleHDCI}
\end{align}
where\begin{equation*}C_1=\left\{
                                                                        \begin{array}{ll}
                                                                          1, & \hbox{if $k$ is even;} \\
                                                                          \frac{\pi}{9}, & \hbox{if $k$ is odd,}
                                                                        \end{array}
                                                                      \right.
\text{ and } C_2=\left\{
                                                                        \begin{array}{ll}
                                                                          \frac{\pi}{9}, & \hbox{if $N-k+1$ is even;} \\
                                                                          1, & \hbox{if $N-k+1$ is odd.}
                                                                        \end{array}
                                                                      \right.
\end{equation*}
\end{pro}
\begin{proof}
Recall the following basic inequalities
\begin{equation*}
 1-\cos x\leq \frac{x^2}{2}, \,\,\forall\, x\geq 0, \,\,1-\cos x\geq\frac{x^2}{\pi}, \,\,\forall\, 0\leq x\leq \frac{\pi}{2}.
\end{equation*}
We only need to consider $k\geq 2$. When $N-k+1$ is odd, we have
$$2-\lambda_{N-k+1}=1-\cos\frac{k\pi}{N}\leq\frac{\pi^2k^2}{2N^2}\leq \frac{\pi^2}{2}\left(\frac{1}{\left\lfloor\frac{N}{k}\right\rfloor}\right)^2=\frac{\pi^2}{2}(1-\overline{h}(k))^2.$$
If $k> \frac{N}{2}$, we have $\overline{h}(k)=0$, and therefore $2-\lambda_{N-k+1}>1=(1-\overline{h}(k))^2$. In the case $k\leq \frac{N}{2}$, we have $\frac{N}{k}\leq\left\lfloor\frac{N}{k}\right\rfloor+1\leq\frac{3}{2}\left\lfloor\frac{k}{N}\right\rfloor$, and
\begin{equation*}
 2-\lambda_{N-k+1}\geq\frac{\pi k^2}{N^2}\geq\frac{4\pi}{9}\left(\frac{1}{\left\lfloor\frac{N}{k}\right\rfloor}\right)^2=\frac{4\pi}{9}(1-\overline{h}(k))^2.
\end{equation*}
This verifies (\ref{exampleHDCI}) for odd $N-k+1$ and in fact also (\ref{exampleHCI}) for even $k$.

When $N-k+1$ is even, we have similarly
\begin{equation*}
 2-\lambda_{N-k+1}=1-\cos\frac{(k-1)\pi}{N}\leq \frac{\pi^2}{2}(1-\overline{h}(k))^2.
\end{equation*}
If $k>\frac{N}{2}+1$, we have $2-\lambda_{N-k+1}>1=(1-\overline{h}(k))^2$. If $k\leq \frac{N}{2}$, then we have $\frac{N}{k-1}\leq \frac{2N}{k}\leq 3\left\lfloor\frac{N}{k}\right\rfloor$ and
\begin{equation*}
 2-\lambda_{N-k+1}\geq \pi\left(\frac{k-1}{N}\right)^2\geq\frac{\pi}{9}\left(\frac{1}{\left\lfloor\frac{N}{k}\right\rfloor}\right)^2=\frac{\pi}{9}(1-\overline{h}(k))^2.
\end{equation*}
 In the case $\frac{N}{2}<k\leq\frac{N}{2}+1$, we have $\overline{h}(k)=0$. Observing that $2-\lambda_{N-k+1}$ is equal to $1$ if $N$ is even, and $1-\cos\frac{(N-1)\pi}{2N}\geq 1-\cos\frac{\pi}{3}>\frac{\pi}{9}$ if $N$ is odd, we have verified (\ref{exampleHDCI}) for even $N-k+1$. This also shows (\ref{exampleHCI}) for odd $k$.
\end{proof}
\begin{remark}
 (\textrm{i}) This example shows that unweighted cycles are a class of graphs for which $\lambda_k$ and $h(k)^2$, as well as $2-\lambda_{N-k+1}$ and $(1-\overline{h}(k))^2$, are equivalent, respectively up to an absolute constant. In fact, one can easily extend this conclusion to weighted cycles with uniformly bounded weights, paying the price that the constant will then depend on the uniform weight bounds. In general, such kind of results for $\lambda_k$ and $h(k)^2$ are called Buser type inequalities. Higher order Buser inequalities for nonnegativly curved graphs are proved in \cite{LiuPeyerimhoff2014}.

(\textrm{ii}) We emphasis that this example also shows that it is possible to expect for certain classes of graphs that $C_1$, $C_2$ can be improved to be $1$ for even $k$ and odd $N-k+1$, respectively. For related discussions about the Cheeger inequality, see Chapter 5 of \cite{MT2006}.

(\textrm{iii}) In \cite{DS1991} (see Example 3.1-3.4 there), Diaconis and Stroock obtained explicit formulas for $h(2)$ of several interesting non-trivial example graphs, including odd cycles.
\end{remark}

\section{Essential spectrum of reversible Markov operators}\label{Application}

In this last section, we discuss an application of the higher-order dual Cheeger inequality (\ref{My2}) to characterize the essential spectrum of a general reversible Markov operator, in the spirit of Miclo \cite{Miclo2013} and F.-Y. Wang \cite{Wang2014}.

Let us start from extending our notations to that abstract setting.
Assume that $(X, \mathcal{F}, \mu)$ is a probability space. We define
$$L^2(X, \mu):=\{f: X\rightarrow \mathbb{R} \text{ measurable, }\int_Xf^2d\mu<+\infty\}.$$
Then let $P: L^2(X, \mu)\rightarrow L^2(X, \mu)$ be a linear operator such that
\begin{equation}\label{Markovian}
P1=1 \text{ and, } \text{ for any } f\in L^2(X, \mu),\,\, f\geq 0 \text{ implies } Pf\geq 0,
\end{equation}
where $1$ stands for the constant function taking the value one. We will also use $1_S$ for a measurable subset $S\subseteq X$ to represent the characteristic function of $S$. The operator $P$ is then called a Markov operator. We will consider a reversible (alternatively called symmetric) Markov operator $P$ with an invariant measure $\mu$. Explicitly, we require for any $f, g\in L^2(X, \mu)$,
\begin{equation}\label{invariant}
\int_Xg(x)Pf(x)d\mu(x)=\int_Xf(x)Pg(x)d\mu(x).
\end{equation}
Actually, there exists a symmetric measure $J$ on $X\times X$ (see e.g. \cite{Wang2014}), such that
\begin{equation}\label{jumpmeasure}
d\mu(x)=\int_{y\in X}J(dx,dy).
\end{equation}
Then we have for any two measurable subsets $A, B$ of $X$,
\begin{equation}\label{Jump}
J(A,B)=(1_A, P1_B)_{\mu},
\end{equation}
where the inner product notion is extended from the previous finite graph setting.
\begin{remark}
The previous weighted finite graph setting can be fitted into this general framework, see \cite{Miclo2013} for a dictionary. However, in this section we only discuss the case that $L^2(X, \mu)$ is infinite dimensional.
\end{remark}

It is known that the spectrum $\sigma(P)$ of the operator $P$ lies in $[-1, 1]$. In the following we will denote the top and bottom of the essential spectrum $\sigma_{\text{ess}}(P)$ of $P$ by
\begin{equation*}
\lambda_{\text{ess}}(P):=\sup \sigma_{\text{ess}}(P) \text{ and } \overline{\lambda}_{\text{ess}}(P):=\inf \sigma_{\text{ess}}(P).
\end{equation*}

 Miclo \cite{Miclo2013} extended the notion of multi-way Cheeger constants and the higher-order Cheeger inequalities to reversible Markov operators among many other different settings. He defines the $k$-way Cheeger constant by
\begin{equation*}
h_P(k):=\inf_{S_1, S_2, \ldots, S_k}\max_{1\leq i\leq k}\frac{(1_{S_i}, P1_{\overline{S}_i})_{\mu}}{\mu(S_i)},
\end{equation*}
where the infimum is taken over all possible $k$-subpartitions, precisely, all collections of $k$ disjoint subsets $S_1, S_2, \ldots, S_k$ such that for each $1\leq i\leq k$, $\mu(S_i)>0$. One then has the monotonicity $h_P(k)\leq h_P(k+1)$.

Adapting Miclo's approximation procedure \cite{Miclo2013}, F.-Y. Wang proved the following characterization of $\lambda_{\text{ess}}(P)$.

\begin{thm}[F.-Y. Wang \cite{Wang2014}] Let $P$ be a reversible Markov operator on $L^2(X, \mu)$. Then
\begin{equation}\label{WangCriterion}
\lambda_{\text{ess}}(P)<1 \Leftrightarrow \sup_{k\geq 1}h_P(k)>0.
\end{equation}
\end{thm}

In their spirit, we extend the $k$-way dual Cheeger constant to the present setting as follows.
\begin{equation*}
\overline{h}_P(k):=\sup_{(A_1, A_2), \ldots, (A_{2k-1}, A_{2k})}\min_{1\leq i\leq k}\frac{2(1_{A_{2i-1}}, P1_{A_{2i}})_{\mu}}{\mu(A_{2i-1}\cup A_{2i})},
\end{equation*}
where the supremum is taken over all possible $k$-sub-bipartitions of $X$, precisely, all collections of $k$ pairs of subsets, $(A_1, A_2), \ldots, (A_{2k-1}, A_{2k})$, where for any $1\leq p\neq q\leq 2k$, $A_p$ and $A_q$ are disjoint and for each $1\leq i\leq k$, $\mu(A_{2i-1}\cup A_{2i})>0$.
Accordingly, we have $\overline{h}_P(k)\geq \overline{h}_P(k+1)$.

Then we have the following relations between $h_P(k)$ and $\overline{h}_P(k)$, extending the previous Proposition \ref{pro1}.
\begin{pro}\label{proMarkovCheeger}
For any $k\geq 1$,
\begin{equation*}
h_P(k)+\overline{h}_P(k)\leq 1.
\end{equation*}
\end{pro}
\begin{proof}
This proposition can be proved in the same way as (\ref{abovecalculation}) proves (\ref{connectionwithDual}), bearing in mind the following fact for a partition $A_1\cup A_2$ of $S\subseteq X$:
\begin{equation*}
\mu(S)=J(S, \overline{S})+2J(A_1, A_2)+J(A_{1}, A_{1})+J(A_{2}, A_{2}),
\end{equation*}
and recalling (\ref{Jump}).
\end{proof}

We prove the following characterization of $\overline{\lambda}_{\text{ess}}(P)$ in terms of the multi-way dual Cheeger constants.
\begin{thm}\label{thmEssentialspectrum}
Let $P$ be a reversible Markov operator on $L^2(X, \mu)$. Then
\begin{equation}\label{Mycriteria}
\overline{\lambda}_{\text{ess}}(P)>-1 \Leftrightarrow \inf_{k\geq 1}\overline{h}_P(k)<1.
\end{equation}
\end{thm}
An immediate corollary is the following.
\begin{coro}\label{coroStrengthenWang}
Let $P$ be a reversible Markov operator on $L^2(X, \mu)$. Then
$$\sup_{k\geq 1}h_P(k)>0 \Leftrightarrow -1<\overline{\lambda}_{\text{ess}}(P)\leq \lambda_{\text{ess}}(P)<1.$$
\end{coro}
This is true, because by Proposition \ref{proMarkovCheeger}, the condition $\inf_{k\geq 1}\overline{h}_P(k)<1$ is weaker than $\sup_{k\geq 1}h_P(k)>0$.

To prove Theorem \ref{thmEssentialspectrum}, we need to extend the higher-order dual Cheeger inequalities to the present setting. Recalling the comments after (\ref{2minusEigen}), the proper operator we should use here is $\overline{L}=I+P$, which is bounded and self-adjoint. Then we can follow \cite{Wang2014} to study
\begin{equation*}
\overline{\lambda}_k:=\sup_{f_1, \ldots,f_{k-1}\in L^2(\mu)}\inf_{\substack{(f,f_i)_{\mu}=0\\\forall 1\leq i\leq k-1}}\frac{(f, \overline{L}f)_{\mu}}{(f, f)_{\mu}}.
\end{equation*}
Define $\overline{\lambda}_{\text{ess}}(\overline{L}):=\inf \sigma_{\text{ess}}(\overline{L})$. Then $\overline{\lambda}_k$ is the $k$-th eigenvalue of $\overline{L}$ if $\overline{\lambda}_k<\overline{\lambda}_{\text{ess}}(\overline{L})$ and $\overline{\lambda}_k=\overline{\lambda}_{\text{ess}}(\overline{L})$ otherwise (see e.g. \cite{Wang2014}).

We can now state the following inequalities.
\begin{thm}\label{thmMarkovHCI}
Let $C$ be the same constant as in (\ref{My2}). Then, for $k\geq 1$,
\begin{equation}\label{HCIMarkov}
\frac{1}{C^2k^6}(1-\overline{h}_P(k))^2\leq \overline{\lambda}_k\leq 2(1-\overline{h}_P(k)).
\end{equation}
\end{thm}
\begin{proof}
The upper bound can be proved by the same technique used in the proof of Theorem \ref{thmLowerbound}. One only needs to keep in mind tha
\begin{equation}
(f, (I+P)f)_{\mu}=2(f,f)_{\mu}-\frac{1}{2}\int_{X\times X}(f(x)-f(y))^2J(dx,dy),
\end{equation}
for any $f\in L^2(X, \mu)$.

For the lower bound, we refer to the proof of Lemma 2.2 in \cite{Wang2014}, because we only need to replace the operator $L=I-P$ there by $\overline{L}$ and use the higher-order dual Cheeger inequalities for the  finite discrete structure there. Basically, the approximation procedure only involves the operator $P$. We also recall here the fact that, on a graph with $N$ vertices, $\overline{\lambda}_k$ of the operator $I+P$ equals $2-\lambda_{N-k+1}$, where $\lambda_{N-k+1}$ is the $(N-k+1)$-th eigenvalue of $\Delta=I-P$.
\end{proof}
\begin{proof}[Proof of Theorem \ref{thmEssentialspectrum}] The proof can be done in the same way as \cite{Wang2014}. For the reader's convenience, we recall it here. First observe
\begin{equation*}
\overline{\lambda}_{\text{ess}}(P)>-1 \Leftrightarrow \overline{\lambda}_{\text{ess}}(\overline{L})>0.
\end{equation*}
If $\overline{\lambda}_{\text{ess}}(\overline{L})>0$, then $\sigma(\overline{L})\cap [0, \overline{\lambda}_{\text{ess}}(\overline{L}))$ is discrete and every eigenvalue in it has finite multiplicity. Therefore, $\overline{\lambda}_k>0$ for large enough $k$. Hence by (\ref{HCIMarkov}), $1-\overline{h}_P(k)>0$ for large $k$.

Otherwise, if $\overline{\lambda}_{\text{ess}}(\overline{L})=0$, then $0\in \sigma_{\text{ess}}(\overline{L})$, and therefore $\overline{\lambda}_k=0$ for all $k$. Now using (\ref{HCIMarkov}) again, we arrive at $\overline{h}_P(k)=1$ for all $k$.
\end{proof}
\begin{remark}
Observe that for this application, the order of $k$ in (\ref{HCIMarkov}) is not important. But we do need the constant in (\ref{My2}) to be universal for any weighted finite graph to derive (\ref{HCIMarkov}) via the approximation procedure in \cite{Miclo2008}, \cite{Wang2014}.
\end{remark}

\section*{Acknowledgements}
The author is very grateful to Norbert Peyerimhoff for many patient enlightening discussions about this topic and in particular for suggesting the example of cycles. The author thanks Feng-Yu Wang for his generous comments, especially to Corollary \ref{coroStrengthenWang}. Thanks also go to Rub\'{e}n S\'{a}nchez-Garc\'{\i}a for the inspiring discussions from which the author learned the topic of multi-way Cheeger inequalities, and to Ioannis Ivrissimtzis for other helpful discussions. Finally, the author acknowledges many useful comments of the anonymous referee.
This work was supported by the EPSRC Grant EP/K016687/1 "Topology, Geometry and Laplacians of Simplicial Complexes".

\end{document}